\documentclass[12pt,a4paper]{article}

\usepackage[utf8x]{inputenc}
\usepackage[english]{babel}
\usepackage[T1]{fontenc}
\usepackage{amsmath}
\usepackage{amsfonts}
\usepackage{mathtools}
\usepackage{amssymb}
\usepackage{makeidx}
\usepackage{verbatim}
\usepackage{graphicx}
\usepackage[left=3cm,right=2cm,top=3cm,bottom=2cm]{geometry}
\usepackage{setspace}
\usepackage{indentfirst}
\usepackage{amsmath}
\usepackage{amssymb}
\usepackage{amsthm}
\usepackage{cite}
\usepackage{enumerate}

\usepackage[colorlinks=true,urlcolor=blue,
citecolor=red,linkcolor=blue,linktocpage,pdfpagelabels,
bookmarksnumbered,bookmarksopen]{hyperref}

\newtheorem{theorem}{Theorem}[section]

\newtheorem{definition}[theorem]{Definition}

\newtheorem{proposition}[theorem]{Proposition}

\newtheorem{lemma}[theorem]{Lemma}

\newtheorem{remark}[theorem]{Remark}

\newtheorem{corollary}[theorem]{Corollary}

\numberwithin{equation}{section}

\usepackage{times, color, xcolor}

\begin{document}

 	\title{Standing waves for nonlinear Hartree type equations: existence and qualitative properties
		}
	\author{
		Eduardo de Souza B\"oer \ \ and \ \  Ederson Moreira dos Santos\thanks{Corresponding author}}
			\date{}	
	\maketitle
	
\noindent \textbf{Abstract:} We consider systems of the form
$$ \left\{ \begin{array}{l}
-\Delta u + u = \dfrac{2p}{p+q}(I_\alpha \ast |v|^{q})|u|^{p-2}u \ \ \textrm{ in } \mathbb{R}^N, \vspace{10pt} \\
-\Delta v + v = \dfrac{2q}{p+q}(I_\alpha \ast |u|^{p})|v|^{q-2}v \ \ \textrm{ in } \mathbb{R}^N, 
\end{array} \right. $$
for $\alpha\in (0, N)$, $\max\left\{\frac{2\alpha}{N}, 1\right\} < p, q < 2^*$ and $\frac{2(N+\alpha)}{N} < p+ q < 2^{*}_{\alpha}$, where $I_\alpha$ denotes the Riesz potential,
\[
2^* = \left\{ \begin{array}{l}\frac{2N}{N-2} \ \ \text{for} \ \ N\geq 3,\\ +\infty \ \ \text{for} \ \ N =1,2, \end{array}\right. \quad \text{and} \quad 2^*_{\alpha} = \left\{ \begin{array}{l}\frac{2(N+\alpha)}{N-2} \ \ \text{for} \ \ N\geq 3,\\ +\infty \ \ \text{for} \ \ N =1,2. \end{array} \right.
\]
This type of systems arises in the study of standing wave solutions for a certain approximation of the Hartree theory for a two-component attractive interaction. We prove existence and some qualitative \linebreak properties for ground state solutions, such as definite sign for each component, radial symmetry and sharp asymptotic decay at infinity, and a regularity/integrability result for the (weak) solutions. Moreover, we show that the straight lines $p+q=\frac{2(N+\alpha)}{N}$ and $ p+ q = 2^{*}_{\alpha}$ are critical for the existence of solutions.

\vspace{0.5 cm}

			
\section{Introduction}			
Consider the two-components Hartree type equations
\begin{equation}\label{hartree}
\begin{cases}
-i \frac{\partial}{\partial t} \varphi =\Delta \varphi + \dfrac{2p}{p+q}\left(I_{\alpha} \ast\left|\psi\right|^q\right)\left|\varphi\right|^{p-2} \varphi \quad \text{in} \ \ (0, +\infty) \times \mathbb{R}^N, \vspace{10pt}\\
-i \frac{\partial}{\partial t} \psi =\Delta \psi + \dfrac{2q}{p+q}\left(I_{\alpha} \ast\left|\varphi\right|^p\right)\left|\psi\right|^{q-2} \psi \quad \text{in} \ \ (0, +\infty) \times \mathbb{R}^N, 
\end{cases}
\end{equation}
where $\varphi,\psi: \mathbb{R}^N \times \mathbb{R} \rightarrow \mathbb{C}$, for $\alpha\in (0, N)$,  $I_\alpha:\mathbb{R}^N \rightarrow \mathbb{R}$ is the Riesz potential, also called Coulumb type potential, defined by
$$
I_\alpha(x)=\dfrac{\Gamma\left(\frac{N-\alpha}{2}\right)}{\Gamma\left(\frac{\alpha}{2}\right)\pi^{\frac{N}{2}}2^\alpha}\,\dfrac{1}{|x|^{N-\alpha}},
$$
$\Gamma$ stands for the Gamma function, and $\ast$ denotes the convolution. 

As presented in \cite{Lieb-Simon, Lions-Hartree, BBL}, Hartree type equations with the Coulomb potential $I_{N-1}$ are used as models to describe the interaction between electrons in the Hartree-Fock theory in quantum chemistry. In addition, systems of the form considered in this paper are used for modeling many physical situations, and for a discussion of how the Hartree type equations appears as a mean-field limit for many-particles boson systems, we refer to \cite{seminaire, lewin}. 

In this paper we study the existence of standing wave solutions $(\varphi, \psi) = e^{it}(u,v)$, with $u,v: \mathbb{R}^N \to \mathbb{R}$, of \eqref{hartree}. This yields that $(u,v)$ must be solution of the system
\begin{equation}\label{P}
\left\{ \begin{array}{ll}
-\Delta u + u = \dfrac{2p}{p+q}(I_\alpha \ast |v|^{q})|u|^{p-2}u \ \ \textrm{ in } \mathbb{R}^N, \vspace{10pt}\\
-\Delta v + v = \dfrac{2q}{p+q}(I_\alpha \ast |u|^{p})|v|^{q-2}v \ \ \textrm{ in } \mathbb{R}^N,
\end{array} \right.
\end{equation}			
also called Choquard type system.

In the case with $p=q$ and $u=v$, system \eqref{P} becomes the Choquard equation
\begin{equation}\label{choquardeq}
-\Delta u + u = (I_\alpha \ast |u|^{p})|u|^{p-2}u \ \ \text{ in } \ \ \mathbb{R}^N.
\end{equation}
This equation arises in physics, with $N=3$, $p=2$ and $\alpha=2$, as a modelling problem for the quantum mechanics of a polaron at rest, first studied by Fröhlich and Pekar in 1954; see \cite{Froehlich1954, Froehlich1939, Wilson1955}. Then, in 1976, Choquard introduced this equation in the study of an electron trapped in its hole \cite{Lieb1977}. Moreover, in 1996, Penrose has derived this equation while discussing about the self gravitational collapse of a quantum-mechanical system in  \cite{Penrose1996}. After these pioneering works, Choquard equations have been extensively studied and the literature is quite vast. We cite the works \cite{Moroz2013, Moroz2017} and the references therein for a more detailed overview, being the former influential to this work. Finally, Choquard equations has a natural extension for $N=2$, here with a logarithmic kernel, derived from the fundamental solution of the Laplacian operator, for which we refer to \cite{Cingolani2016}.

Much less results are known regarding the study of standing wave solutions of \eqref{hartree}. We cite \cite{Bhattarai}, where existence and stability of standing wave solutions of \eqref{hartree} are studied in the case with $p=q$ and $N=2,3$, and \cite{Wang-Shi} for the study of standing waves for a related system. We also mention that Zhang and Chen \cite{Chen2023a} studied
\begin{equation}\label{eq26}
\left\{ \begin{array}{ll}
-\Delta u + A(x)u = \dfrac{2p}{p+q}(I_\alpha \ast |v|^{q})|u|^{p-2}u \ \ \textrm{ in } \mathbb{R}^N, \vspace{10pt}\\
-\Delta v + B(x)v = \dfrac{2q}{p+q}(I_\alpha \ast |u|^{p})|v|^{q-2}v \ \ \textrm{ in } \mathbb{R}^N, 
\end{array} \right.
\end{equation}
with $\frac{N+\alpha}{N}<p, q < c_{\alpha}$, for $c_{\alpha}=\frac{N+\alpha}{N-2}$ if $N\geq 3$ and $c_{\alpha}= \infty$ if $N=1,2$, and proved the existence of positive ground state solutions for \eqref{eq26}. They considered the positive potentials $A$ and $B$ as bounded and periodic (in particular, $A = B = 1$). However, up to our knowledge, no qualitative properties of these solutions (such as regularity/integrability, symmetry, asymptotic decay) are available in the literature.

In this paper we present a more complete picture about the existence of solutions and ground state solutions of \eqref{P}. Namely, we expand the region on the $(p, q)$-plan where positive solutions to \eqref{P} exist and we prove: the definite sign for each component $u$ and $v$, radial symmetry and sharp asymptotic decay at infinity for ground state solutions of \eqref{P}; a regularity result for the (weak) solutions of \eqref{P}; we prove that the straight lines $p+q=\frac{2(N+\alpha)}{N}$ and $ p+ q = 2^{*}_{\alpha}$ are critical for the existence of solutions, in the sense that if $(p,q)$ is on or below the first, or if it is on or above the second, then no nontrivial solution of \eqref{P} exists. We also prove some general results that can be useful in other frameworks. For example, the Brezis-Lieb type result for double sequences as given by Proposition \ref{l14}.

We start establishing our hypotheses on $p, q$ for the existence of solutions for \eqref{P}. Our bounds for $p$ and $q$ take into consideration the both versions of Hardy-Littlewood-Sobolev theorems (Propositions \ref{t1} and \ref{t2}) and Sobolev embeddings, namely
\begin{equation}\label{H1}
\tag{H1}
\max\left\{\dfrac{2\alpha}{N}, 1\right\} < p, q < 2^{*} \mbox{ \ \ and \ \ } \dfrac{2(N+\alpha)}{N}< p+ q < 2^{*}_{\alpha},
\end{equation}
where
\[
2^* = \left\{ \begin{array}{l}\frac{2N}{N-2} \ \ \text{for} \ \ N\geq 3,\\ +\infty \ \ \text{for} \ \ N =1,2, \end{array}\right. \quad \text{and} \quad 2^*_{\alpha} = \left\{ \begin{array}{l}\frac{2(N+\alpha)}{N-2} \ \ \text{for} \ \ N\geq 3,\\ +\infty \ \ \text{for} \ \ N =1,2. \end{array} \right.
\]
At this point we stress that in the case with $p=q$, $u=v$, where system \eqref{P} becomes the Choquard equation, condition \eqref{H1} becomes 
\[
\frac{N-2}{N+\alpha} < \frac{1}{p} < \frac{N}{N+\alpha},
\] 
which appears in \cite{Moroz2013} for the existence of solutions to \eqref{choquardeq}. 

We state next our main results, and we refer to Definition \ref{def-sol} ahead for the precise definition of (weak) solution of \eqref{P}.

\begin{theorem}\label{t3}
Let $0<\alpha< N$, $N\geq 1$, and $p, q$ be as in \eqref{H1}. Then system \eqref{P} has a ground state solution $(u, v)$ such that $u>0$ and $v>0$ in $\mathbb{R}^N$. Moreover, if $(u, v)$ is a ground state solution of \eqref{P}, then $|u| >0$ and $|v|>0$ in $\mathbb{R}^N$.
\end{theorem}

\begin{remark} If $(u,v)$ is a solution of \eqref{P}, then $(u,-v)$, $(-u,v)$ and $(-u, -v)$ are also solutions of \eqref{P}. In particular, under \eqref{H1}, system \eqref{P} admits ground state solutions where one of the components is positive and the other is negative.
\end{remark}

At Figure \ref{pic3} we compare our results regarding existence (Theorem \ref{t3}) with those available in the literature, in \cite[Corollary 1.3]{Chen2023a}.
\begin{figure}[!h]
\begin{center}
 \includegraphics[width=7cm]{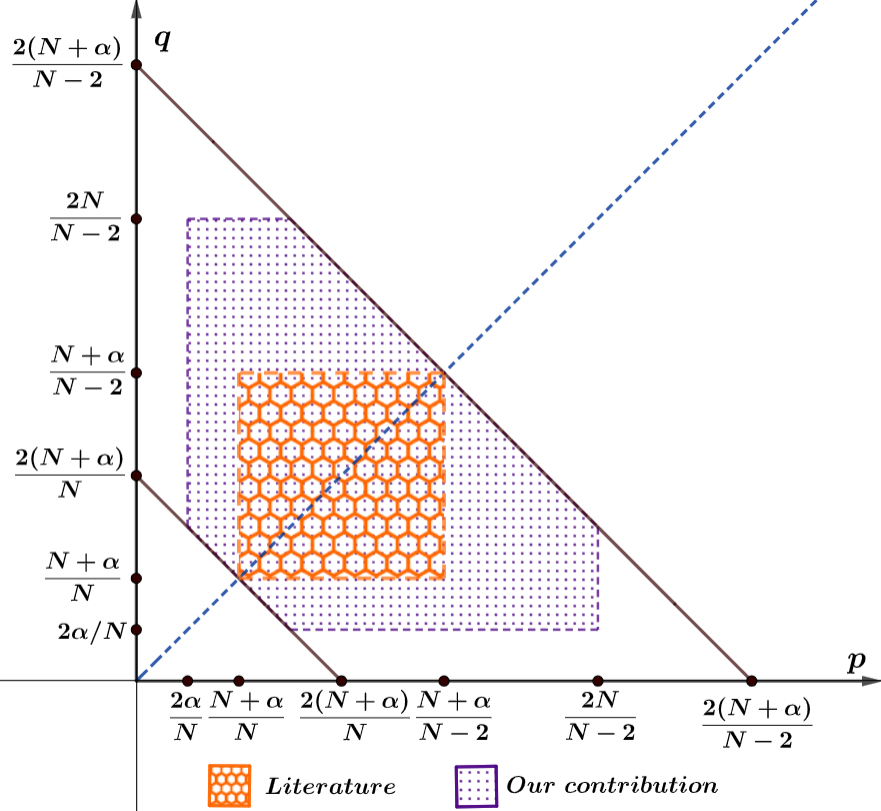}   
\end{center}
\caption{Region on the $(p,q)$-plane corresponding to our existence results (Theorem \ref{t3}) in comparison to those in the literature. Here $N=3$ and $\alpha=1$.}\label{pic3}
\end{figure}

Using an approach based on polarization, we prove that any ground state solution of \eqref{P}, with positive components, is radially symmetric and radially decreasing with respect to a common center of symmetry.

\begin{theorem}\label{t5}
Let $0<\alpha< N$, $N\geq 1$, and $p, q$ be as in \eqref{H1}. Let $(u, v)$ be a ground state solution for \eqref{P} such that $u>0$ and $v>0$ in $\mathbb{R}^N$. Then there exists $x_0 \in \mathbb{R}^N$ such that $u$ and $v$ are both radially symmetric and radially decreasing with respect to $x_0$.
\end{theorem}

Now we turn to a result regarding the regularity/integrability of solutions of \eqref{P}. In comparison with the scalar case \cite[Proposition 4.1]{Moroz2013}, see also \cite{Ma-Zao-2010} and \cite{Cingolani-Clapp-Secchi}, the bootstrap arguments are more technical. Besides depending on the position of the point $(p,q)$ in the plane, in the iterating process, the regularity of $u$ interferes on the regularity of $v$ and vice versa. Apart from having a long proof, this regularity result is an important ingredient for the proof of sharp asymptotic decay for ground state solutions.

\begin{theorem}\label{t10novo}
Let $0<\alpha< N$, $N\geq 1$, $p, q$ as in \eqref{H1}, and $(u, v)\in E$ be a solution of \eqref{P}. 
\begin{itemize}
    \item[(a)] If
    $$
    q \geq \dfrac{\alpha}{N}p + 1 \mbox{ \ \ \ and \ \ \ } q \leq \dfrac{N}{\alpha}p - \dfrac{N}{\alpha},
    $$
    then $u\in W^{2, r}(\mathbb{R}^N)$ and $v\in W^{2, h}(\mathbb{R}^N)$, for all $1< r, h <+\infty$.

    \item[(b)] If
    $$
    p < \dfrac{N}{N-\alpha} \mbox{\ \ \ and \ \ \ } q < \dfrac{N}{N-\alpha},
    $$
    then $u\in W^{2, r}(\mathbb{R}^N)$ and $v\in W^{2, h}(\mathbb{R}^N)$, for all $1< r^\ast < r <+\infty$ and $1< h^\ast< h< + \infty$, respectively, where
    \begin{align*}
    &r^\ast = \dfrac{pN}{p(2N-\alpha) -N} \mbox{ \ \ \ and \ \ \ } h^\ast = \dfrac{qN}{q(2N-\alpha) -N}, \mbox{ \ \ \ if \ \ \ } p,q\geq\frac{2N}{2N-\alpha},\vspace{10pt}\\
    &r^\ast = \dfrac{(2-p)N}{N- \alpha} \mbox{ \ \ \ and \ \ \ } h^\ast = \dfrac{(2-p)qN}{(2q+p-2)N - 2q\alpha}, \mbox{ \ \ \ if \ \ \ } p<\frac{2N}{2N-\alpha},\vspace{10pt}\\
    &r^\ast = \dfrac{(2-q)pN}{(2p+q-2)N - 2p\alpha} \mbox{ \ \ \ and \ \ \ } h^\ast = \dfrac{(2-q)N}{N- \alpha} , \mbox{ \ \ \ if \ \ \ } q<\frac{2N}{2N-\alpha}.
    \end{align*}
    \item[(c)] If
     $$
    q \geq \dfrac{N}{N-\alpha} \mbox{ \ \ \ and \ \ \ } q > \dfrac{N}{\alpha}p - \dfrac{N}{\alpha},
    $$
    then $u\in W^{2, r}(\mathbb{R}^N)$ and $v\in W^{2, h}(\mathbb{R}^N)$, for all $1< r^{\ast \ast} < r <+\infty$ and $1\leq h^{\ast \ast}< h<  + \infty$, respectively, where
     \begin{align*}
    &r^{\ast \ast} = \dfrac{pqN}{(pq+p-1)N-q\alpha} \mbox{ \ \ \ and \ \ \ } h^{\ast \ast} = 1, \mbox{ \ \ \ if \ \ \ } \dfrac{2-p}{N-\alpha} \leq \dfrac{pq}{N +\alpha q},\vspace{10pt}\\
    &r^{\ast \ast} = \dfrac{(2-p)N}{N- \alpha} \mbox{ \ \ \ and \ \ \ } h^{\ast \ast} = \dfrac{(2-p)qN}{(2q+p-2)N - 2q\alpha}, \mbox{ \ \ \ if \ \ \ } \dfrac{2-p}{N-\alpha} > \dfrac{pq}{N +\alpha q}.
    \end{align*}
    \item[(d)] If
     $$
    q < \dfrac{\alpha}{N}p + 1 \mbox{ \ \ \ and \ \ \ } p \geq \dfrac{N}{N-\alpha},
    $$
    then $u\in W^{2, r}(\mathbb{R}^N)$ and $v\in W^{2, h}(\mathbb{R}^N)$, for all $1\leq r^{\ast \ast\ast} < r <+\infty$ and $1<h^{\ast \ast \ast} < h< + \infty$, respectively, where
    \begin{align*}
    &r^{\ast \ast \ast} =1 \mbox{ \ \ \ and \ \ \ } h^{\ast \ast \ast} = \dfrac{pqN}{(pq+q-1)N-p\alpha}, \mbox{ \ \ \ if \ \ \ } \dfrac{2-q}{N-\alpha} < \dfrac{pq}{N +\alpha p},\vspace{10pt}\\
    &r^{\ast \ast \ast} = \dfrac{(2-q)pN}{(2p+q-2)N - 2p\alpha}  \mbox{ \ \ \ and \ \ \ } h^{\ast \ast \ast} = \dfrac{(2-q)N}{N- \alpha}, \mbox{ \ \ \ if \ \ \ } \dfrac{2-q}{N-\alpha} \geq \dfrac{pq}{N +\alpha p}.
    \end{align*}
\end{itemize}
In any case, $u, v\in L^{\infty}(\mathbb{R}^N) \cap C^2(\mathbb{R}^N)$, $u \in C^{\infty}(\mathbb{R}^N \setminus u^{-1}(\{0\}))$, $v \in C^{\infty}(\mathbb{R}^N \setminus v^{-1}(\{0\}))$ and satisfy \eqref{P} pointwise in the classical sense. 
\end{theorem}
\begin{figure}[!h]
\begin{center}
 \includegraphics[width=8cm]{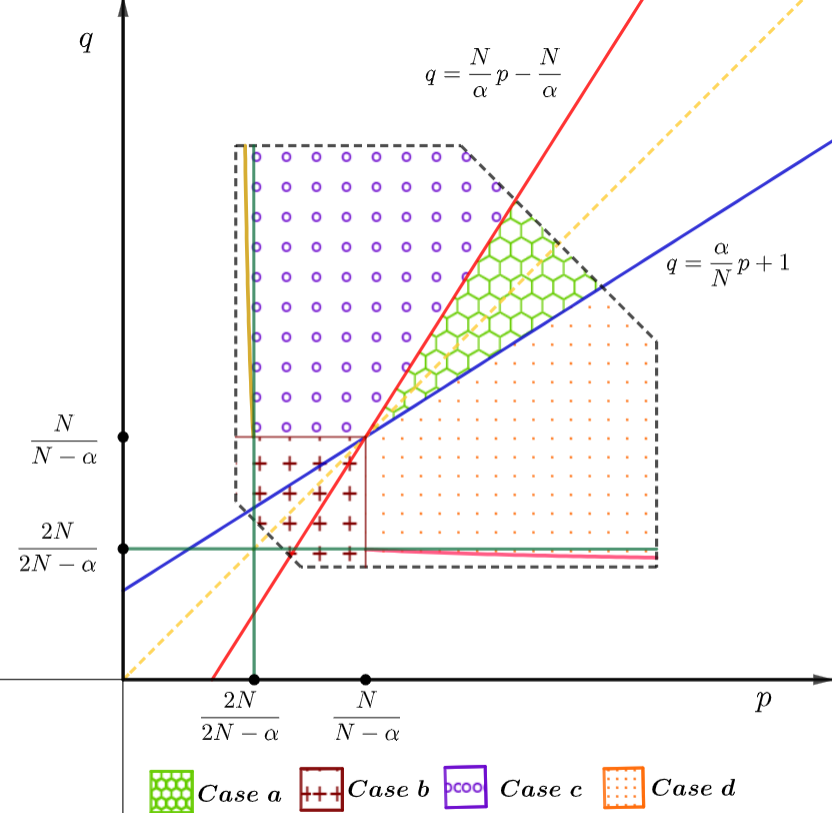}   
\end{center}
\caption{Regions on the $(p,q)$-plane corresponding to cases (a)-(d) of Theorem \ref{t10novo}. In this picture, $N=3$, $\alpha=1.9$.}\label{figreg}
\end{figure} 

\begin{remark}
At Theorem \ref{t10novo} (c), observe that the condition $\frac{2-p}{N-\alpha} > \frac{pq}{N +\alpha q}$ says that $p$ is small, and that $(p,q)$ is below a hyperbola that meets $(\frac{2N}{2N-\alpha}, \frac{N}{N-\alpha})$ at the common boundary of regions (c) and (b). Moreover, observe that $r^{\ast} = r^{\ast \ast}$ and $h^{\ast} = h^{\ast \ast}$ when $p$ is small at regions (b) and (c). A similar remark holds for the respective case involving regions (b) and (d).  
\end{remark}

Next, strongly based on Theorem \ref{t10novo}, we prove the sharp decay for ground state solutions of \eqref{P}.

\begin{theorem}\label{decaimentosol}
Let $N\geq 1$, $\alpha\in (0, N)$ and assume \eqref{H1}. In addition:
\begin{itemize}
    \item[(i)] If $\min\{p,q\} < \max\{p,q\} =2$, then also assume $\min\{p,q\} > \max\left\{ \frac{2N}{2N-\alpha}, \frac{2(\alpha+1)}{N+1}\right\}$;

\item[(ii)] If $\max\{p,q\} <2$, then also assume $\min\{p,q\} >  \frac{2N}{2N-\alpha}$.

\end{itemize}
Under these conditions, if $(u, v)$ is a ground state solution for \eqref{P} with $u>0$ and $v>0$, then:
\begin{itemize}
    \item[(a)] For $p>2$, $\displaystyle \lim\limits_{|x|\rightarrow + \infty} u(x)|x|^{\frac{N-1}{2}}e^{|x|} \in (0, +\infty)$.
    
        \item[(b)] For $q>2$, $\lim\limits_{|x|\rightarrow + \infty} v(x)|x|^{\frac{N-1}{2}}e^{|x|} \in (0, +\infty) .$

\item[(c)] For $p=2$, 
\begin{align*}
&\lim\limits_{|x|\rightarrow + \infty} u(x)|x|^{\frac{N-1}{2}}\exp \left(\displaystyle \int_{\mathcal{A}}^{|x|}\sqrt{1-\dfrac{\mathcal{A}^{N-\alpha}}{s^{N-\alpha}}} ds\right) \in (0, +\infty), \ \ \text{with} \vspace{10pt}\\
&\mathcal{A}^{N-\alpha}=\dfrac{2q}{p+q}\dfrac{\Gamma\left(\frac{N-\alpha}{2}\right)}{\Gamma\left(\frac{\alpha}{2}\right)\pi^{\frac{N}{2}}2^{\alpha}}\displaystyle \int_{\mathbb{R}^N} |v|^q. \end{align*} 

\item[(d)] For $q=2$,
\begin{align*}
&\lim\limits_{|x|\rightarrow + \infty} v(x)|x|^{\frac{N-1}{2}}\exp \left(\displaystyle \int_{\mathcal{B}}^{|x|}\sqrt{1-\dfrac{\mathcal{B}^{N-\alpha}}{s^{N-\alpha}}} ds\right) \in (0, +\infty), \ \ \text{with} \vspace{10pt}\\ &\mathcal{B}^{N-\alpha}=\dfrac{2p}{p+q}\dfrac{\Gamma\left(\frac{N-\alpha}{2}\right)}{\Gamma\left(\frac{\alpha}{2}\right)\pi^{\frac{N}{2}}2^{\alpha}}\displaystyle \int_{\mathbb{R}^N} |u|^p.\end{align*}

\item[(e)] For $p<2$
$$
\lim\limits_{|x|\rightarrow +\infty} (u(x))^{2-p}|x|^{N-\alpha}= \dfrac{2p}{p+q}\dfrac{\Gamma\left(\frac{N-\alpha}{2}\right)}{\Gamma\left(\frac{\alpha}{2}\right)\pi^{\frac{N}{2}}2^\alpha}\int_{\mathbb{R}^N} |v|^q dx.$$

\item[(f)] For $q<2$ 
$$\lim\limits_{|x|\rightarrow +\infty} (v(x))^{2-q}|x|^{N-\alpha}= \dfrac{2q}{p+q}\dfrac{\Gamma\left(\frac{N-\alpha}{2}\right)}{\Gamma\left(\frac{\alpha}{2}\right)\pi^{\frac{N}{2}}2^\alpha}\int_{\mathbb{R}^N} |u|^p dx.
$$
\end{itemize}
\end{theorem}

\begin{remark}
Theorem \ref{decaimentosol} deserves some comments. First, observe that in case $\max\{p,q\} >2$ or $p=q=2$ no extra condition (besides \eqref{H1}) is needed to get the asymptotic decay for $u$ and $v$ at infinity. Regarding condition at (i) and (ii), taking into consideration the lower critical exponent $\frac{N+\alpha}{N}$ that appears for the Choquard equation, we stress that
\[
\max\left\{ \frac{2N}{2N-\alpha}, \frac{2(\alpha+1)}{N+1}\right\} < \frac{N+\alpha}{N}.
\]
In addition, in contrast with the scalar case, as in \cite[Theorem 4]{Moroz2013}, which implies that every ground state solution of \eqref{choquardeq} is $L^1(\mathbb{R}^N)$, there are ground state solution of \eqref{P} such that one of the components is not $L^1(\mathbb{R}^N)$. For instance, this is the case with $p>2$ and $q \leq \frac{N+\alpha}{N}$.
\end{remark}

\medbreak
Finally, we prove a non-existence result, which shows that \eqref{H1} is a natural condition for the existence of solutions. We recall that the Choquard equation \eqref{choquardeq}, with $p>1$, has two critical exponents for the existence of solution, namely, if $pN \leq N+\alpha$ or $p(N-2)\geq N+\alpha$, then no nontrivial solution for \eqref{choquardeq} exists; see \cite[Theorem 2]{Moroz2013}. Theorem below says that the corresponding result for the Hartree system \eqref{P} is the existence of two critical straight lines, namely
\[
(p+q)N = 2(N+\alpha) \quad \text{and} \quad (p+q)(N-2) = 2(N+\alpha).
\]

\begin{theorem}\label{t4}
Assume $N\geq 1$, $\alpha\in (0, N)$, $p>1$ and $q>1$. If $(u,v)$ is a weak solution of \eqref{P}, with $u\in H^{1}(\mathbb{R}^N) \cap W_{loc}^{2, 2}(\mathbb{R}^N)\cap W^{1, \theta_1 p}(\mathbb{R}^N)$, $v\in H^{1}(\mathbb{R}^N) \cap W_{loc}^{2, 2}(\mathbb{R}^N)\cap W^{1, \theta_2 q}(\mathbb{R}^N)$ for some $\theta_1>1, \theta_2> 1$ with $\frac{1}{\theta_1}+ \frac{1}{\theta_ 2} = \frac{N+\alpha}{N}$ and such that
\[
(p+q)N \leq 2(N+\alpha) \quad \text{or} \quad (p+q)(N-2) \geq 2(N+\alpha) ,
\]
then $u= v=0$ in $\mathbb{R}^N$.
\end{theorem}

We illustrate at Figure \ref{pic2} our results regarding existence (Theorem \ref{t3}), non-existence (Theorem \ref{t4}), as well as the regions on the $(p,q)$-plane that are not covered by our (non-)existence results. 

\begin{figure}[!h]
\begin{center}
 \includegraphics[width=6cm]{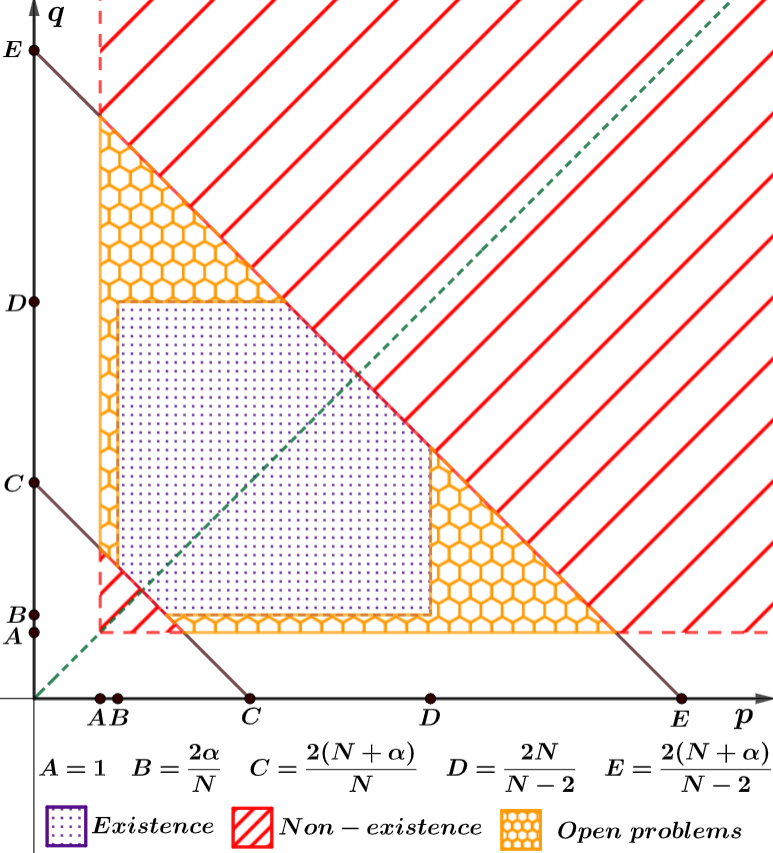}   
\end{center}
\caption{Regions on the $(p,q)$-plane corresponding to Theorem \ref{t3}, Theorem \ref{t4} and open problems. Here $N=3$ and $\alpha=1.9$.}\label{pic2}
\end{figure}

To end this introduction we describe how this paper is organized. In Section \ref{sec-var-frame} we describe in details how hypothesis \eqref{H1} enters for setting the appropriate definition of (weak) solution to \eqref{P} and the variational framework. Section \ref{sec-nehari} is devoted to the analysis of the Nehari manifold and the mountain pass geometry of the energy functional associated to \eqref{P}. In Section \ref{sec-exist} we prove that \eqref{P} has at least one positive ground state solution, while in Sections \ref{sec-symm} and \ref{sec-decay} we prove their symmetry and asymptotic decay, respectively. At Section \ref{sec-reg} we provide a regularity result for any (weak) solution of \eqref{P}. Finally, at Section \ref{sec-NE}, based on a Poho\v{z}aev-type identity, we establish regions on the $(p, q)$-plane where no nontrivial solution of \eqref{P} exists.

\section{Variational framework}\label{sec-var-frame}

In this section we set the variational framework associated to problem \eqref{P} and prove some technical results. We consider $H^{1}(\mathbb{R}^N)$ endowed with its usual norm $\|\cdot\|=\left(\|\,|\nabla \cdot|\,\|_{2}^{2}+\|\cdot\|_{2}^{2}\right)^{1/2}$, where $\|\cdot\|_\tau$ stands for the usual norm for the $L^\tau(\mathbb{R}^N)$ space, and $E=H^{1}(\mathbb{R}^N) \times H^{1}(\mathbb{R}^N)$ with the norm \linebreak $\|(u, v)\|_{E}=\left(\|u\|^2 + \|v\|^2\right)^{1/2}$.

Since the systems in this paper depend on Riesz potentials, we recall the well known \linebreak Hardy-Littlewood-Sobolev inequality, \cite[Theorem 4.3]{Lieb2001}, and we will refer to it as \eqref{HLS1} for short. 

\begin{proposition}\label{t1}
Let $N\geq1$, $\theta_1, \theta_2 > 1$, $\alpha\in(0,N)$ be such that $\frac{1}{\theta_1} + \frac{1}{\theta_2} = \frac{N+\alpha}{N}$. Consider $f\in L^{\theta_1}(\mathbb{R}^N)$ and $g \in L^{\theta_2}(\mathbb{R}^N)$. Then, there exists a constant $K_0=K_0(N, \alpha, \theta_1)$, independent of $f$ and $g$, such that
\begin{equation}\label{HLS1} \tag{HLS1}
\left| \int_{\mathbb{R}^N} \int_{\mathbb{R}^N} f(x)|x-y|^{-(N-\alpha)}g(y) dx dy \right| \leq K_0 \|f\|_{\theta_1} \|g\|_{\theta_2}.
\end{equation} 
\end{proposition}

Moreover, as in \cite[eq. (1.3)]{Moroz2013}, the following result, also known as Hardy-Littlewood-Sobolev inequality plays an important role in our arguments.  

\begin{proposition}\label{t2} 
Let $N\geq 1$, $\alpha\in (0, N)$ and $s\in \left(1, \frac{N}{\alpha}\right)$. Then, for any $\varphi \in L^{s}(\mathbb{R}^N)$, $(I_\alpha \ast \varphi) \in L^{\frac{Ns}{N-\alpha s}}(\mathbb{R}^N)$ and
\begin{equation}\label{HLS2} \tag{HLS2}
\int_{\mathbb{R}^N} |I_\alpha \ast \varphi|^{\frac{Ns}{N-\alpha s}} dx \leq K_1 \left(\int_{\mathbb{R}^N}  |\varphi|^s dx \right)^{\frac{N}{N-\alpha s}},
\end{equation}
where $K_1=K_1(N, \alpha, s)$ is a positive constant.
\end{proposition}

For using variational methods to treat \eqref{P}, from Proposition \ref{t1} and the Sobolev embeddings of $H^1(\mathbb{R}^N) \hookrightarrow L^{r}(\mathbb{R}^N)$ for $\frac{N-2}{2N}\leq \frac{1}{r} \leq \frac{1}{2}$, we want to classify the numbers $p>0$ and $q>0$ such that there exist $\theta_1 , \theta_2$ satisfying the following conditions:
\begin{equation}\label{novaeq}
\theta_1, \theta_2 >1, \quad \dfrac{N-2}{2N} < \frac{1}{\theta_1 p} < \frac{1}{2}, \ \ \ \dfrac{N-2}{2N} < \frac{1}{\theta_2 q} < \frac{1}{2} \mbox{ \ \ \ and \ \ \ } \dfrac{1}{\theta_1}+\dfrac{1}{\theta_2}=\dfrac{N+\alpha}{N}, 
\end{equation}
which are equivalent to
\[
\theta_1, \theta_2 >1, \quad \frac{N-2}{2N}p < \frac{1}{\theta_1} < \frac{p}{2}, \quad \frac{N-2}{2N}q < \frac{1}{\theta_2} < \frac{q}{2} \mbox{ \ \ \ and \ \ \ } \dfrac{1}{\theta_1}+\dfrac{1}{\theta_2}=\dfrac{N+\alpha}{N}.
\]
Then, from direct calculation, the precise conditions on $p>0$ and $q>0$ for the existence of such $\theta_1$ and $\theta_2$ are
\begin{equation}\label{H2}
\tag{H2}
\frac{N-2}{2N} < \frac{1}{p}\, , \frac{1}{q} < \frac{N}{2\alpha} \quad \text{with} \quad \dfrac{N-2}{2(N+\alpha)} < \dfrac{1}{p+q} < \dfrac{N}{2(N+\alpha)}.  
\end{equation}

Throughout in this paper, when \eqref{H2} is assumed, $\theta_1$ and $\theta_2$ as in \eqref{novaeq} are chosen and fixed. Also observe that \eqref{H2}, combined with $p>1$ and $q>1$, reads \eqref{H1}.

\begin{remark}\label{thetas}
One of the contributions of this paper is the relaxation of the condition on $(p,q)$ for proving the existence of solution to \eqref{P}. For using variational methods, the key ingredient is the existence of $\theta_1$ and $\theta_2$ as in \eqref{novaeq}. Since $\theta_2$ depends on $\theta_1$, the precise condition on $\theta_1$ for this purpose is that
\begin{equation}\label{conditionthetaum}
\max\left\{ \dfrac{\alpha}{N}, \dfrac{p(N-2)}{2N}, \dfrac{N+\alpha}{N} - \dfrac{q}{2} \right\} < \dfrac{1}{\theta_1} < \min\left\{ 1, \dfrac{p}{2}, \dfrac{N + \alpha}{N} - \dfrac{q(N-2)}{2N} \right\}.
\end{equation}
Indeed, observe that \eqref{H2} is equivalent to the existence of $\theta_1$ as in \eqref{conditionthetaum}. When choosing $\theta_1 = \frac{2N}{N+\alpha}$, and so $\theta_1 = \theta_2$, one imposes extra restrictions on $(p,q)$, namely $\frac{N-2}{N+\alpha}< \frac{1}{p}, \frac{1}{q}< \frac{N}{N+\alpha}$, as in \cite{Chen2023a}. Another natural par $(\theta_1, \theta_2)$ would be $\theta_1 = \frac{N}{N+\alpha} \frac{p+q}{p}$ and $\theta_2 = \frac{N}{N+\alpha} \frac{p+q}{q}$, which in turn would require the extra condition
\[
\dfrac{\alpha}{N} < \dfrac{q}{p} < \dfrac{N}{\alpha}.
\]
So, in this paper, supposing \eqref{H1}, hence \eqref{H2}, although not given explicitly, there exist $\theta_1, \theta_2$ as in \eqref{novaeq}.
\end{remark}

\begin{remark}\label{obs1}
Let $\alpha \in (0,N)$, $p,q$ satisfying \eqref{H2}, $\theta_1$ and $\theta_2$ be as in \eqref{novaeq}. From \eqref{HLS1} and the Sobolev embeddings, then for all $u, v\in H^{1}(\mathbb{R}^N)$,
\begin{equation}\label{eq5}
\int_{\mathbb{R}^N} (I_\alpha \ast |u|^p)|v|^q dx \leq K_0 \|u\|_{\theta_1 p}^{p}\|v\|_{\theta_2 q}^{q}.
\end{equation}
If in addition $p>1$ and $q>1$, namely if we assume \eqref{H1}, then for all $u,v,\varphi, \psi \in H^{1}(\mathbb{R}^N)$,
\begin{equation}\label{eq6}
\left.
\begin{array}{rcl}
\displaystyle\int_{\mathbb{R}^N} (I_\alpha \ast |v|^q)|u|^{p-2} u \varphi dx &\leq& \displaystyle K_0 \|u\|_{\theta_1 p}^{p-1}\|\varphi\|_{\theta_1p}\|v\|_{\theta_2 q}^{q},\vspace{10pt}\\
\displaystyle\int_{\mathbb{R}^N} (I_\alpha \ast |u|^p)|v|^{q-2} v \psi dx &\leq&\displaystyle K_0 \|u\|_{\theta_1 p}^{p}\|v\|_{\theta_2 q}^{q-1}\|\psi\|_{\theta_2 q}.
\end{array}
\right\}
\end{equation}
\end{remark}

\begin{definition}\label{def-sol}
At this point, assuming \eqref{H1}, we can define the Euler–Lagrange functional $I: E \rightarrow \mathbb{R}$,
\begin{align}\label{eq1}
I(u,v)& =\dfrac{1}{2}\int_{\mathbb{R}^N} [|\nabla u|^2 +|u|^2 +|\nabla v|^2 +|v|^2] dx - \dfrac{2}{p+q}\int_{\mathbb{R}^N} \int_{\mathbb{R}^N} \dfrac{|u(x)|^p |v(y)|^q}{|x-y|^{N-\alpha}} dx dy \\  \nonumber
& = \dfrac{1}{2}\|(u, v)\|_{E}^{2}- \dfrac{2}{p+q}\int_{\mathbb{R}^N} (I_\alpha \ast |u|^p)|v|^q dx,
\end{align}
which is $C^1(E; \mathbb{R})$, whose derivative at $(u,v)$ is given by
\begin{align}\label{eq2}
I'(u, v)(\varphi, \psi) & = \int_{\mathbb{R}^N} [\nabla u \nabla \varphi + u \varphi + \nabla v \nabla \psi + v \psi ] dx - \dfrac{2q}{p+q}\int_{\mathbb{R}^N} (I_\alpha \ast |u|^p)|v|^{q-2}v \psi dx  \\ \nonumber
& - \dfrac{2p}{p+q}\int_{\mathbb{R}^N} (I_\alpha \ast |v|^q)|u|^{p-2}u \varphi dx, \quad \forall \, (\varphi, \psi) \in E.
\end{align}
A pair $(u,v) \in E$ is a solution of \eqref{P} if it is a critical point of $I$. Moreover, a solution $(U,V) \in E\backslash\{(0,0\}$ of \eqref{P} is called a ground state solution if
\[
I(U,V) = \min\{\,I(u,v); \ (u,v) \ \text{is a nontrivial solution of \eqref{P}}\,\}.
\]
\end{definition}

\section{Nehari manifold and mountain pass geometry}\label{sec-nehari}

Throughout in this section we assume \eqref{H1} and so $\theta_1$ and $\theta_2$ as in \eqref{novaeq} are chosen and fixed. Here we explore some properties of the Nehari manifold associated to \eqref{P} and show that $I$ has the mountain pass geometry around $(0,0)$. The proof for some of these results are classical and, for this reason, are omitted. 

Consider the auxiliary functional $P: E \rightarrow \mathbb{R}$ given by
\begin{equation}\label{eq3}
P(u, v)=I'(u, v)(u, v)=\|(u, v)\|_{E}^{2}-2\int_{\mathbb{R}^N} (I_\alpha \ast |u|^p)|v|^q dx,
\end{equation}
the Nehari manifold
\begin{equation}\label{eq4}
\mathcal{N} = \{(u, v)\in E\setminus \{(0, 0)\} ; P(u, v)=0\}, 
\end{equation}
and define the real values
\begin{align*}
&c_{mp}=\inf\limits_{\gamma\in\Gamma} \max\limits_{t\in [0, 1]} I(\gamma(t)), \ \  c_{\mathcal{N}} = \inf_{(u,v) \in \mathcal{N}} I(u,v), \ \  \overline{c} = \inf\limits_{(u, v)\in E\setminus \{(0, 0)\}} \max\limits_{t>0} I(tu, tv) \ \ \text{and} \vspace{5pt} \\ &\widetilde{c} = \left(\frac{1}{2}-\frac{1}{p+q}\right) \inf_{(u,v)\in E, u \neq 0, v \neq 0}  \left[ \frac{\|(u,v)\|_E^2}{\left(2\int_{\mathbb{R}^N} (I_{\alpha} \ast |u|^p)|v|^q \right)^{2/(p+q)}}\right]^{(p+q)/(p+q-2)}.
\end{align*}
where $\Gamma = \{\gamma:[0, 1]\rightarrow E \ ; \gamma \mbox{ is continuous }, \gamma(0)=0 \mbox{ and } I(\gamma(1))<0\}$.

\begin{lemma}\label{l3}
Let $w = (u,v)\in E$ be such that $u\neq0$ and $v \neq 0$. Then, there exists a unique $\overline{t}_w  >0$ such that $\overline{t}_w w\in \mathcal{N}$. Moreover:
\begin{enumerate}[(i)]
\item $\displaystyle \max_{t \in [0, \infty)} I(tw) = I(\overline{t}_w w)$.
\item $
\overline{t}_w = \left[\dfrac{\|(u, v)\|_{E}^{2}}{2\int_{\mathbb{R}^N} (I_\alpha \ast |u|^p)|v|^q dx} \right]^{\frac{1}{p+q-2}}.
$
\item $\displaystyle I(\overline{t}_w w) = \left(\frac{1}{2}-\frac{1}{p+q}\right)  \left[ \frac{\|(u,v)\|_E^2}{\left(2\int_{\mathbb{R}^N} (I_{\alpha} \ast |u|^p)|v|^q \right)^{2/(p+q)}}\right]^{(p+q)/(p+q-2)}$.
\item $w = (u,v)\in \mathcal{O} \mapsto \overline{t}_w \in \mathbb{R}$ is continuous, where $\mathcal{O}$ stands for the open set $\{ (u,v) \in E ; u\neq 0 \ \ \text{and} \ \ v \neq 0\}$.
\item $c_{mp}= \overline{c} = \widetilde{c}= c_{\mathcal{N}}>0$.
\end{enumerate}
\end{lemma}
\begin{proof}
It follows from direct computations that are similar as, for example, in \cite[Chapter 4]{Willem1996}.
\end{proof}

\begin{lemma}\label{l6}
There exists a constant $K_2>0$ such that $\|(u, v)\|_E \geq K_2 > 0$, for all $(u, v)\in \mathcal{N}$.
\end{lemma}
\begin{proof}
Let $(u, v)\in \mathcal{N}$. Then from \eqref{eq5}, the embeddings $H^1(\mathbb{R}^N) \hookrightarrow L^{\theta_1p}(\mathbb{R}^N)$, $H^1(\mathbb{R}^N) \hookrightarrow L^{\theta_21}(\mathbb{R}^N)$, we infer that
$$
\|(u, v)\|_{E}^{2} = 2 \int_{\mathbb{R}^N} (I_\alpha \ast |u|^p)|v|^q dx 
\leq K_0 \|u\|_{\theta_1 p}^{p}\|v\|_{\theta_2 q}^{q} \leq C \|(u, v)\|_{E}^{p+q}. 
$$
The desired inequality follows from the fact that $p+q>2$.
\end{proof}

\begin{lemma}\label{l7}
There exists $\rho>0$ such that
\begin{equation}\label{eq10}
m_\beta = \inf\{I(u, v) \ ; (u, v)\in E, \|(u, v)\|_E = \beta \} > 0, \mbox{ \ \ for any \ \ } 0 < \beta \leq \rho
\end{equation}
and 
\begin{equation}\label{eq11}
n_\beta = \inf\{P(u, v) \ ; (u, v)\in E, \|(u, v)\|_E = \beta \} > 0, \mbox{ \ \ for any \ \ } 0 < \beta \leq \rho.
\end{equation}
\end{lemma}
\begin{proof}
From \eqref{eq5}, the embeddings $H^1(\mathbb{R}^N) \hookrightarrow L^{\theta_1p}(\mathbb{R}^N)$, $H^1(\mathbb{R}^N) \hookrightarrow L^{\theta_21}(\mathbb{R}^N)$, we infer that
\begin{align*}
&I(u, v) \geq \dfrac{1}{2}\|(u, v)\|_{E}^{2}\left[\, 1-C_1\|(u, v)\|_{E}^{p+q-2}\,\right] \mbox{ \ \ \ and \ \ \ }\vspace{10pt}\\  &P(u, v) \geq \|(u, v)\|_{E}^{2}\left[\, 1-C_2\|(u, v)\|_{E}^{p+q-2}\, \right].
\end{align*}
Then, for $\rho > 0$ sufficiently small, the desired inequalities follow from the fact that $p+q>2$.
\end{proof}

\begin{lemma}\label{l8}
The real number $c_\mathcal{N}$ is positive.
\end{lemma}
\begin{proof}
It is a direct consequence of Lemma \ref{l3}(i) and Lemmas \ref{l6} and \ref{l7}. 
\end{proof}

\begin{lemma}\label{l9}
$\mathcal{N}$ is a $C^1$ manifold.
\end{lemma}
\begin{proof}
We show that $P'(u, v)\neq 0$ for all $(u, v)\in \mathcal{N}$. Observe that,
$$
P'(u,v)(u, v)= 2\|(u, v)\|_{E}^{2}-2(p+q)\int_{\mathbb{R}^N} (I_\alpha \ast |u|^p)|v|^q dx = 2(1-p-q)\|(u, v)\|_{E}^{2} <0, 
$$
for all $(u, v)\in \mathcal{N}$, since $p+q>2$.
\end{proof}

Finally, we show that $\mathcal{N}$ is a natural constraint for $I$.

\begin{lemma}\label{l10}
Every critical point of $I\big\vert_{\mathcal{N}}$ is a critical point of $I$ in $E$.
\end{lemma}
\begin{proof}
Let $(u, v)$ be a critical point of $I\big\vert_{\mathcal{N}}$, that is, $(u, v)\in \mathcal{N}$ and $(I\big\vert_{\mathcal{N}})'(u, v)=0$. From Lagrange multipliers, there exists $\lambda\in \mathbb{R}$ such that
$
I'(u, v)= \lambda P'(u, v)
$. Then $0=I'(u, v)(u, v)=\lambda P'(u, v)(u, v)$, which implies $\lambda=0$, because $P'(u, v)(u, v)\neq 0$, by Lemma \ref{l9}. Therefore, $I'(u, v)=0$.
\end{proof}

\section{Existence of positive solutions}\label{sec-exist}

In this section we prove Theorem \ref{t3}. We begin recalling a variant of the classical result of Brezis and Lieb \cite{BrezisLieb1983} as presented in \cite[Lemma 2.5]{Moroz2013}.

\begin{lemma}\label{l13}
Let $s\in [1, +\infty)$ and $(u_{n})$ a bounded sequence in $L^{s}(\mathbb{R}^N)$. If $u_{n} \rightarrow u$ a.e. in $\mathbb{R}^N$, then for all $r \in [1, s]$, 
$$
\lim\limits_{n\rightarrow + \infty} \int_{\mathbb{R}^N} |\,|u_{n}|^r -|u_{n} - u|^r -|u|^r|^{\frac{s}{r}} = 0.
$$
\end{lemma}

We use this lemma to get a similar result involving the Riesz potential and double sequences. Such result, in the scalar version, is proved in \cite[Lemma 2.4]{Moroz2013}; see also \cite[p. 90]{Lions-Hartree}. For the double sequence case, under stronger conditions on the sequences and on $p$ and $q$, a similar result is proved in \cite[Lemma 2.4]{Chen2023a}.

\begin{proposition}\label{l14}
Let $N \geq 1$, $\alpha \in (0, N)$, $p, q \geq 1$, $\theta_1, \theta_2>1$ be such that $\frac{1}{\theta_1}+ \frac{1}{\theta_2} = \frac{N+\alpha}{N}$. Suppose that $(u_n) \subset L^{\theta_1p}(\mathbb{R}^N)$, $(v_n) \subset L^{\theta_2q}(\mathbb{R}^N)$ are bounded sequences such that $u_n \to u$ and $v_n \to v$ almost everywhere in $\mathbb{R}^N$. Then
$$
\lim\limits_{n\rightarrow +\infty} \left[\int_{\mathbb{R}^N} (I_\alpha \ast |u_{n}|^p)|v_{n}|^q dx - \int_{\mathbb{R}^N} (I_\alpha \ast |u_{n} -u|^p)|v_{n} -v|^q dx \right] = \int_{\mathbb{R}^N} (I_\alpha \ast |u|^p)|v|^q dx.
$$
\end{proposition}
\begin{proof}
For each $n\in \mathbb{N}$, write
\begin{align}\label{eq12}
& \int_{\mathbb{R}^N} (I_\alpha \ast |u_{n}|^p)|v_{n}|^q dx - \int_{\mathbb{R}^N} (I_\alpha \ast |u_{n} -u|^p)|v_{n} -v|^q dx = \\ \nonumber
& \int_{\mathbb{R}^N} [I_\alpha \ast (|u_{n}|^p - |u_{n} -u|^p)]|v_{n}|^q dx - \int_{\mathbb{R}^N} [(I_\alpha \ast |u_{n} -u|^p)](|v_{n} -v|^q -|v_{n}|^q) dx.
\end{align}
Lemma \ref{l13} guarantees that $|u_{n}|^p - |u_{n} - u|^p \to |u|^p$ in $L^{\theta_1}(\mathbb{R}^N)$. Also observe that 
$$\frac{1}{\theta_1} = \frac{N+\alpha}{N} - \frac{1}{\theta_2}> \frac{N+\alpha}{N} - 1 = \frac{\alpha}{N}.$$
Then, from \eqref{HLS2}, we infer that $I_\alpha \ast (|u_{n}|^p - |u_{n} -u|^p) \to I_{\alpha}\ast |u|^p$ in $L^{N\theta_1/(N - \alpha\theta_1)}(\mathbb{R}^N)$. On the other hand, $|v_{n}|^q \rightharpoonup |v|^q$ in $L^{\theta_2}(\mathbb{R}^N)$ and
$$
\frac{N-\alpha\theta_1}{N \theta_1} + \frac{1}{\theta_2} = \frac{N-\alpha\theta_1}{N \theta_1} - \frac{1}{\theta_1} + \frac{N+\alpha}{N} =1.
$$
Therefore,
\begin{equation}\label{eqprim}
\int_{\mathbb{R}^N} [I_\alpha \ast (|u_{n}|^p - |u_{n} -u|^p )]|v_{n}|^q dx \to  \int_{\mathbb{R}^N} (I_\alpha \ast |u|^p)|v|^q dx, \ \ \text{as} \ \ n \to \infty.
\end{equation}
Next, observe that
\[
\int_{\mathbb{R}^N} [(I_\alpha \ast |u_{n} -u|^p)](|v_{n} -v|^q -|v_{n}|^q) dx = \int_{\mathbb{R}^N} [I_\alpha \ast (|v_{n} -v|^q -|v_{n}|^q)] |u_{n} -u|^pdx.
\]

Arguing as before, $I_\alpha \ast (|v_{n} -v|^q -|v_{n}|^q) \to I_{\alpha}\ast |v|^q$ in $L^{N\theta_2/(N-\alpha \theta_2)}(\mathbb{R}^N)$ and $|u_n-u|^p \rightharpoonup 0$ in $L^{\theta_1}(\mathbb{R}^N)$, with
\[
\frac{N-\alpha\theta_2}{N \theta_2} + \frac{1}{\theta_1} = 1,
\]
yielding that
\begin{equation}\label{eqseg}
\int_{\mathbb{R}^N} [(I_\alpha \ast u_{n} -u|^p)](|v_{n} -v|^q -|v_{n}|^q) dx  \to 0 \ \ \text{as} \ \ n \to \infty.
\end{equation}

Therefore, the result follows from \eqref{eq12}, \eqref{eqprim} and \eqref{eqseg}.
\end{proof}

\begin{corollary}\label{c4}
Let $N \geq 1$, $\alpha \in (0, N)$, $p, q \geq 1$, $\theta_1, \theta_2>1$ such that $\frac{1}{\theta_1}+ \frac{1}{\theta_2} = \frac{N+\alpha}{N}$. If $u_{n} \rightarrow u$ in $L^{\theta_1 p}(\mathbb{R}^N)$ and $v_{n} \rightharpoonup v$ in $L^{\theta_2 q}(\mathbb{R}^N)$, or $u_{n} \rightharpoonup u$ in $L^{\theta_1 p}(\mathbb{R}^N)$ and $v_{n} \rightarrow v$ in $L^{\theta_2 q}(\mathbb{R}^N)$, then
$$
\lim\limits_{n\rightarrow +\infty} \int_{\mathbb{R}^N} (I_\alpha \ast |u_{n}|^p)|v_{n}|^q dx = \int_{\mathbb{R}^N} (I_\alpha \ast |u|^p)|v|^q dx.
$$
\end{corollary}
\begin{proof}
It is a direct consequence of Proposition \ref{l14} and \eqref{HLS1}.
\end{proof}

Next we recall some useful results on Schwarz symmetrization, which can be found in \cite[Section 3 in Chapter 6]{Kavian1993} and \cite[Chapter 3]{Lieb2001}. A measurable function $u:\mathbb{R}^N \to \mathbb{R}$ is said to vanish at infinity, if $\textrm{meas}(\{ x: |u(x)|>t\})$ is finite for all  $t>0$. Given a measurable function $u:\mathbb{R}^N \to \mathbb{R}$ vanishing at infinity, we denote by $u^\ast$ its Schwarz symmetrization (symmetric decreasing rearrangement).

\begin{proposition}\label{t8} Let $G:[0, +\infty)\rightarrow [0, +\infty)$ be a continuous increasing function such that $G(0)=0$.
\begin{enumerate}[(a)]
\item If $u:\mathbb{R}^N \to \mathbb{R}$ is nonnegative, measurable and vanishes at infinity, then 
$$
\int_{\mathbb{R}^N} G(u^\ast(x)) dx = \int_{\mathbb{R}^N} G(u(x)) dx.
$$
\item If $u \in L^r(\mathbb{R}^N)$, with $1 \leq r < \infty$, is a nonnegative function, then  $\|u^\ast\|_r = \|u\|_r$.
\item If $u\in H^{1}(\mathbb{R}^N)$ is a nonnegative function, then $u^\ast \in H^{1}(\mathbb{R}^N)$ and
$
\|\nabla u^\ast\|_{2}^{2} \leq \|\nabla u\|_{2}^{2}.
$
\item Let $f, g, h: \mathbb{R^N} \to \mathbb{R}$ be nonnegative measurable functions vanishing at infinite and set
\[
J(f, g, h) = \int_{\mathbb{R}^N} \int_{\mathbb{R}^N} f(x)g(x-y)h(y) dx dy.
\]
Then,
\[
J(f, g, h) \leq J(f^\ast, g^\ast, h^\ast).
\]
\end{enumerate}

\end{proposition}

\medbreak

\begin{proof}[\bf{Proof of Theorem \ref{t3}}] Here we borrow some ideas from the proof of \cite[Proposition 2.2]{Moroz2013}. For $(u,v) \in E$ with $u \neq 0$ and $v \neq 0$, set
\[
S_{\alpha, p,q} (u,v) = \frac{\|(u,v)\|_E^2}{\left(2\int_{\mathbb{R}^N} (I_{\alpha} \ast |u|^p)|v|^q \right)^{\frac{2}{p+q}}}.
\]

Let $(u_{n}, v_{n})\subset \mathcal{N}$ such that $I(u_{n}, v_{n})\rightarrow c_\mathcal{N}$. First of all,  observe that $P(u, v) =  P(|u|, |v|)$ and $I(u, v)=I(|u|, |v|)$, for all $(u,v) \in E$. Thus, without loss of generality, we may assume that $u_{n} \geq 0$ and $v_{n} \geq 0$ for all $n\in \mathbb{N}$. 

For each $n\in \mathbb{N}$, consider the unique $t_{n}^{\ast}>0$ such that $(t_{n}^{\ast}u_{n}^{\ast}, t_{n}^{\ast}v_{n}^{\ast})\in \mathcal{N}$. From Proposition \ref{t8}, $I(t_{n}^{\ast}u_{n}^{\ast}, t_{n}^{\ast}v_{n}^{\ast}) \leq I(t_{n}^{\ast}u_{n}, t_{n}^{\ast}v_{n})$ and from Lemma \ref{l3}, $I(t_{n}^{\ast}u_{n}, t_{n}^{\ast}v_{n})\leq I(u_{n}, v_{n})$, for all $n\in \mathbb{N}$. Consequently, besides being nonnegative, we may also assume that $u_{n}, v_{n}$ are radially symmetric and radially decreasing for all $n \in \mathbb{N}$.

\noindent \textbf{Claim 1.} The sequence $(u_{n}, v_{n})$ is bounded in $E$.

It follows from 
\[
c_\mathcal{N} + o(1) = I(u_{n}, v_{n}) - \dfrac{1}{p+q}P(u_{n}, v_{n}) = \left(\dfrac{1}{2}-\dfrac{1}{p+q}\right)\|(u_{n}, v_{n})\|_{E}^{2}, \quad \text{as }\ \ n \to \infty.
\]
and the fact that $p+q>2$.

\medbreak

Since $H_{rad}^{1}(\mathbb{R}^N)$ is a weakly closed subspace of $H^{1}(\mathbb{R}^N)$, there exist $u,v \in H^1_{rad}(\mathbb{R}^N)$ such that, up to a subsequence, $u_{n} \rightharpoonup u$ and $v_{n} \rightharpoonup v$ in $H^{1}(\mathbb{R}^N)$ and almost everywhere in $\mathbb{R^N}$. 

\medbreak

\noindent \textbf{Claim 2.} The week limits are such that $u\neq0$ and $v \neq0$.

First, from Lemma \ref{l6}, eq. \eqref{eq5} and \cite[Lemma 2.3]{Moroz2013} (see also \cite[Lemma I.1]{Lions1985II}), observe that
\begin{align*}
   K^2_2 &\leq \|(u_{n}, v_{n})\|_{E}^{2} = 2 \int_{\mathbb{R}^N} (I_\alpha \ast |u_{n}|^p)|v_{n}|^q dx \leq C \|u_n\|_{\theta_1 p}^{p}\|v_n\|_{\theta_2 q}^{q}\\ & \leq  C \left( \left(\int_{B_1} |u_n|^{\theta_1p}\right)^{1- \frac{2}{\theta_1p}} \|u_n\|^2\right)^{1/\theta_1}\left( \left(\int_{B_1} |v_n|^{\theta_2q}\right)^{1- \frac{2}{\theta_2q}} \|v_n\|^2\right)^{1/\theta_2}.
\end{align*}
From this inequality, since $(u_n)$ and $(v_n)$ are bounded in $H^1(\mathbb{R})$, we infer that 
\[
\inf_n \int_{B_1}|u_n|^{\theta_1p} >0 \ \ \text{ and } \ \ \inf_n \int_{B_1}|v_n|^{\theta_2q} >0.
\]
Then, from the compact embeddings $H^1(B_1) \subset L^{\theta_1p}(B_1)$ and $H^1(B_1) \subset L^{\theta_2q}(B_2)$, we conclude that $u \neq 0$ and $v \neq 0$.
\medbreak

From Lemma \ref{l3}, let $\overline{t}>0$ be such that $(\overline{t}u, \overline{t}v) \in \mathcal{N}$. Then, from the explicit formulas in Lemma \ref{l3} (ii) and (iii), with $k_{p,q} = \frac{1}{2} - \frac{1}{p+q}$ we infer that
\[
c_{\mathcal{N}}  \leq I(\overline{t}u, \overline{t}v) = k_{p,q} \left[ \frac{\|(u,v)\|_E^2}{\left(2\int_{\mathbb{R}^N} (I_{\alpha} \ast |u|^p)|v|^q \right)^{\frac{2}{p+q}}}\right]^{\frac{p+q}{p+q-2}} = k_{p,q} \left[S_{\alpha, p,q}(u,v)\right]^{\frac{p+q}{p+q-2}}.
\]

From this inequality, the classical Brezis-Lieb result relating the norms $\|(u_n, v_n)\|_{E}^2$ and $\|(u,v)\|_E^2$ and Lemma \ref{l14}, we infer that
\begin{align}
    c_{\mathcal{N}} &\leq k_{p,q}\lim_{n \to \infty}\left[ S_{\alpha,p,q}(u_n,v_n) \left(\frac{\int_{\mathbb{R}^N} (I_{\alpha} \ast |u_n|^p)|v_n|^q}{2\int_{\mathbb{R}^N} (I_{\alpha} \ast |u|^p)|v|^q}  \right)^{\frac{2}{p+q}} \right. \vspace{10pt} \nonumber\\ & \hspace{1cm} \left.- S_{\alpha,p,q}(u_n-u, v_n-v)\left(\frac{\int_{\mathbb{R}^N} (I_{\alpha} \ast |u_n-u|^p)|v_n-v|^q}{2\int_{\mathbb{R}^N} (I_{\alpha} \ast |u|^p)|v|^q}  \right)^{\frac{2}{p+q}}\right]^{\frac{p+q}{p+q-2}} \vspace{10pt}\nonumber\\
    & \leq c_{\mathcal{N}} \liminf_{n \to \infty} \left[\left(\frac{\int_{\mathbb{R}^N} (I_{\alpha} \ast |u_n|^p)|v_n|^q}{2\int_{\mathbb{R}^N} (I_{\alpha} \ast |u|^p)|v|^q}  \right)^{\frac{2}{p+q}} -  \left(\frac{\int_{\mathbb{R}^N} (I_{\alpha} \ast |u_n-u|^p)|v_n-v|^q}{2\int_{\mathbb{R}^N} (I_{\alpha} \ast |u|^p)|v|^q}  \right)^{\frac{2}{p+q}}\right]^{\frac{p+q}{p+q-2}}.\label{novaeqseq}
\end{align}
Since $p+q>2$, by Lemma \ref{l14}, if $\displaystyle\int_{\mathbb{R}^N} (I_{\alpha}\ast |u|^p) |v|^q < \liminf_{n \to \infty}\int_{\mathbb{R}^N} (I_{\alpha}\ast |u_n|^p) |v_n|^q$, then 
\[
\liminf_{n \to \infty}\left(\frac{\int_{\mathbb{R}^N} (I_{\alpha} \ast |u_n|^p)|v_n|^q}{2\int_{\mathbb{R}^N} (I_{\alpha} \ast |u|^p)|v|^q}  \right)^{\frac{2}{p+q}} -  \left(\frac{\int_{\mathbb{R}^N} (I_{\alpha} \ast |u_n-u|^p)|v_n-v|^q}{2\int_{\mathbb{R}^N} (I_{\alpha} \ast |u|^p)|v|^q}  \right)^{\frac{2}{p+q}}< 1,
\]
which would lead \eqref{novaeqseq} to a contradiction.

Therefore,
\begin{equation}\label{importante}
\|(u, v)\|_{E}^{2} \leq \liminf\limits_{n\rightarrow +\infty} \|(u_{n}, v_{n})\|_{E}^{2} = 2 \liminf\limits_{n\rightarrow +\infty} \int_{\mathbb{R}^N} (I_\alpha \ast |u_{n}|^p)|v_{n}|^q dx = 2 \int_{\mathbb{R}^N} (I_\alpha \ast |u|^p)|v|^q dx.
\end{equation}

Thus, from \eqref{importante} and Lemma \ref{l3}, we conclude that $\overline{t} \leq 1$ and
\begin{align}
c_\mathcal{N} & \leq I(\overline{t} u, \overline{t} v) = \left(\dfrac{1}{2}-\dfrac{1}{p+q}\right)\overline{t}^{ \, p+q} \int_{\mathbb{R}^N} (I_\alpha \ast |u|^p)|v|^q dx \\ \nonumber
& \leq \left(\dfrac{1}{2}-\dfrac{1}{p+q}\right)\int_{\mathbb{R}^N} (I_\alpha \ast |u|^p)|v|^q dx \\ \nonumber
& = \liminf\limits_{n\rightarrow +\infty} \left(\dfrac{1}{2}-\dfrac{1}{p+q}\right) \int_{\mathbb{R}^N} (I_\alpha \ast |u_{n}|^p)|v_{n}|^q dx = \liminf\limits_{n\rightarrow +\infty} I(u_{n}, v_{n}) = c_\mathcal{N}. 
\end{align}
Consequently, $\overline{t} =1$ and $I(u, v)=c_\mathcal{N}$, with $u\geq 0$ and $v\geq 0$. Then, by the maximum principle, e.g. as in \cite[Proposition 2.1]{Van2008}, we infer that $u>0$ and $v>0$ almost everywhere in $\mathbb{R}^N$.

\medbreak
If $(u, v)$ is a ground state solution of \eqref{P} so is $(|u|, |v|)$, and the result follows from the maximum principle, again as in \cite[Proposition 2.1]{Van2008}.
\end{proof}

\section{Symmetry of positive ground state solutions}\label{sec-symm}

In this section we prove that any positive ground state solution is radial, namely Theorem \ref{t5}. To this aim, we use some arguments based on polarization as in \cite{Brock1999, Moroz2013}. 

Let $H\subset \mathbb{R}^N$ be a closed half-space and $\sigma_H$ be the reflection with respect to $\partial H$. We define the polarization of a function $u:\mathbb{R}^N \rightarrow \mathbb{R}$, denoted by $u^H : \mathbb{R}^N \rightarrow \mathbb{R}$, as
$$
u^H(x)=\left\{\begin{array}{ll}
\max\{u(x), u(\sigma_H(x))\}, \mbox{ \ \ if } x\in H, \\
\min\{u(x), u(\sigma_H(x))\}, \mbox{ \ \ \ if } x\notin H.
\end{array}\right.
$$ 
We start with the following characterization of a radial function, as presented in \cite[Lemma 5.4]{Moroz2013}.

\begin{lemma}\label{l15}
Let $s\geq 1$ and $u\in L^{s}(\mathbb{R}^N)$. If $u\geq 0$ and for every closed half-space $H\subset \mathbb{R}^N$, $u^H = u$ or $u^H = u \circ \sigma_H$, then there exists $x_0\in \mathbb{R}^N$ and $v:(0, +\infty)\rightarrow \mathbb{R}$ a nondecreasing function such that for almost every $x\in \mathbb{R}^N$, $u(x)=v(|x-x_0|)$.
\end{lemma}

That said, to prove Theorem \ref{t5}, we show that any (positive) ground state solution $(u,v)$ satisfies the hypotheses in Lemma \ref{l15} and that the center of symmetry $x_0$ is the same for both components $u$ and $v$.

\begin{lemma}\label{l16}
Let $N\geq1$, $\theta_1, \theta_2 > 1$, $\alpha\in(0,N)$ be such that $\frac{1}{\theta_1} + \frac{N-\alpha}{N} + \frac{1}{\theta_2} = 2$. Consider $f\in L^{\theta_1}(\mathbb{R}^N)$, $g \in L^{\theta_2}(\mathbb{R}^N)$ and $H\subset \mathbb{R}^N$ any closed half-space. If $f> 0$, $g > 0$ a.e. in $\mathbb{R^N}$ and
$$
\int_{\mathbb{R}^N} \int_{\mathbb{R}^N} \dfrac{f(x)g(y)}{|x-y|^{N-\alpha}} dx dy \geq \int_{\mathbb{R}^N} \int_{\mathbb{R}^N} \dfrac{f^H(x)g^H(y)}{|x-y|^{N-\alpha}} dx dy, 
$$
then either $u^H = u$ and $v^H = v$ or $u^H = u\circ \sigma_H $ and $v^H = v\circ \sigma_H$.
\end{lemma}
\begin{proof}
First of all, observe that both integrals are finite by \eqref{HLS1}. Moreover, we can write
\begin{align*}
\int_{\mathbb{R}^N} \int_{\mathbb{R}^N} \dfrac{f(x)g(y)}{|x-y|^{N-\alpha}} dx dy &= \int_{H} \int_{H} \dfrac{f(x)g(y)+f(\sigma_H(x))g(\sigma_H(y))}{|x-y|^{N-\alpha}} dx dy  \\ & +\int_{H} \int_{H}\dfrac{f(x)g(\sigma_H(y))+f(\sigma_H(x))g(y)}{|x-\sigma_H(y)|^{N-\alpha}} dx dy
\end{align*}
and
\begin{align*}
\int_{\mathbb{R}^N} \int_{\mathbb{R}^N} \dfrac{f^H(x)g^H(y)}{|x-y|^{N-\alpha}} dx dy &= \int_{H} \int_{H} \dfrac{f^H(x)g^H(y)+f^H(\sigma_H(x))g^H(\sigma_H(y))}{|x-y|^{N-\alpha}}dx dy\\& + \int_{H} \int_{H}\dfrac{f^H(x)g^H(\sigma_H(y))+f^H(\sigma_H(x))g^H(y)}{|x-\sigma_H(y)|^{N-\alpha}} dx dy.
\end{align*}

On the other hand, for all $x, y\in H$, for $u, v \geq 0$, $\alpha < N$ and the definition of $u^H$ and $v^H$, we have
\begin{align}\label{despontualref}
& \dfrac{f(x)g(y)+f(\sigma_H(x))g(\sigma_H(y))}{|x-y|^{N-\alpha}} + \dfrac{f(x)g(\sigma_H(y))+f(\sigma_H(x))g(y)}{|x-\sigma_H(y)|^{N-\alpha}} \\
& \leq \dfrac{f^H(x)g^H(y)+f^H(\sigma_H(x))g^H(\sigma_H(y))}{|x-y|^{N-\alpha}} + \dfrac{f^H(x)g^H(\sigma_H(y))+f^H(\sigma_H(x))g^H(y)}{|x-\sigma_H(y)|^{N-\alpha}} \nonumber,
\end{align}
thanks to $|x-y| \leq |x-\sigma_H(y)|$ for all $x,y \in H$.
Thus, from hypothesis, the equality must hold a.e. in \eqref{despontualref}. Considering the sets $\Omega_1 = \{ x\in H \ ; u^H(x) = u(x)\}$, $\Omega_2 = \{ x\in H \ ; u^H(x) = u(\sigma_H(x))\}$, $\Omega_3 = \{ y\in H \ ; v^H(y) = v(y)\}$ and $\Omega_4 = \{ y\in H \ ; v^H(y) = v(\sigma_H(y))\}$, it follows that the equality holds in \eqref{despontualref} a.e if, and only if, either $u^H = u$ and $v^H = v$ or $u^H = u\circ \sigma_H $ and $v^H = v\circ \sigma_H$.
\end{proof}

\begin{corollary}\label{c5}
Let $f$ and $g$ be as in the conditions of Lemma \ref{l16}. Then, $f$ and $g$ are radial functions with a common center of symmetry.
\end{corollary}
\begin{proof}
From Lemmas \ref{l15} and \ref{l16}, there are $x_0, x_1\in \mathbb{R}^N$ such that $f$ is radially symmetric with center at $x_0$ and $v$ is axially symmetric with center at $x_1$. Suppose that $x_0 \neq x_1$. In this case, consider the straight segment $L$ joining $x_0$ and $x_1$ and the closed half-space $H$ such that $x_0 \in H$, $\partial H$ is orthogonal to $L$ and intersects its middle point.  Thus, from Lemma \ref{l16}, either $f^H = f$ and $g^H = g$ or $f^H = f\circ \sigma_H $ and $g^H = g\circ \sigma_H$. If the first one happens, then $g$ must be constant and this contradicts the integrability of $g$. The second case induces a similar contradiction. Therefore, $x_0 = x_1$.
\end{proof}

It remains to prove that any positive ground state solution $(u,v)$ of \eqref{P} is such that $f= u^p$ and $g = v^q$ verify the hypotheses of Lemma \ref{l16}, which is proved below.

\begin{lemma}\label{l17}
Let $(u, v)$ be a positive ground state solution of \eqref{P} and $H$ be any closed half-space of $\mathbb{R}^N$. Then $(u^H, v^H)$ is also a ground state solution of \eqref{P}.
\end{lemma}
\begin{proof}
Since $(u, v)\in \mathcal{N}$, 
$$
\|(u, v)\|_{E}^{2} =2 \int_{\mathbb{R}^N} (I_\alpha \ast |u|^p)|v|^q dx.
$$
Let $u^H$ and $v^H$ be the polarization of $u$ and $v$, respectively. Since $u^H \neq 0$ and $v^H \neq 0$,  from Lemma \ref{l3}, there exists a unique $\overline{t}_H > 0$ such that $(\overline{t}_H u^H, \overline{t}_H v^H)\in \mathcal{N}$. Thus, again by Lemma \ref{l3},
\begin{align*}
I(\overline{t}_H u^H, \overline{t}_H v^H)= &\left(\dfrac{1}{2}-\dfrac{1}{p+q}\right)(\overline{t}_H)^2 \|(u^H, v^H)\|_{E}^{2}\\=&\left(\dfrac{1}{2}-\dfrac{1}{p+q}\right)\dfrac{(\|(u^H, v^H)\|_{E}^{2})^{\frac{p+q}{p+q-2}}}{\left(2\int_{\mathbb{R}^N} (I_\alpha \ast |u^H|^p)|v^H|^q dx \right)^{\frac{2}{p+q-2}}}.
\end{align*}
Since
$$
\|(u, v)\|_{E}^{2}=\|(u^H, v^H)\|_{E}^{2} \mbox{ \ \ \ and \ \ \ } \int_{\mathbb{R}^N} (I_\alpha \ast |u|^p)|v|^q dx \leq \int_{\mathbb{R}^N} (I_\alpha \ast |u^H|^p)|v^H|^q dx,
$$
we infer that $\overline{t}_H \leq 1$ and
\begin{align*}
c_\mathcal{N} & = \left(\dfrac{1}{2}-\dfrac{1}{p+q}\right) \|(u, v)\|_{E}^{2}\leq I(\overline{t}_H u^H, \overline{t}_H v^H) = \left(\dfrac{1}{2}-\dfrac{1}{p+q}\right)\dfrac{(\|(u^H, v^H)\|_{E}^{2})^{\frac{p+q}{p+q-2}}}{\left(2\int_{\mathbb{R}^N} (I_\alpha \ast |u^H|^p)|v^H|^q dx \right)^{\frac{2}{p+q-2}}} \\
& \leq \left(\dfrac{1}{2}-\dfrac{1}{p+q}\right)\dfrac{(\|(u, v)\|_{E}^{2})^{\frac{p+q}{p+q-2}}}{\left(2\int_{\mathbb{R}^N} (I_\alpha \ast |u|^p)|v|^q dx \right)^{\frac{2}{p+q-2}}} = \left(\dfrac{1}{2}-\dfrac{1}{p+q}\right) \|(u, v)\|_{E}^{2} = c_\mathcal{N}.
\end{align*}
Therefore, we conclude that $\overline{t}_H = 1$ and $I(u^H, v^H)=I(u,v)$. Moreover,
\begin{equation}\label{novaigualdade}
\int_{\mathbb{R}^N} (I_\alpha \ast |u|^p)|v|^q dx = \int_{\mathbb{R}^N} (I_\alpha \ast |u^H|^p)|v^H|^q dx. \qedhere
\end{equation}
\end{proof}

\begin{proof}[\bf{Proof of Theorem \ref{t5}}]
Let $(u, v)$ be a positive ground state solution of \eqref{P}. From Lemma \ref{l16}, Corollary \ref{c5} and \eqref{novaigualdade}, we infer that there exists $x_0 \in \mathbb{R}^N$ such that $u$ and $v$ are both radial and radially decreasing with respect to $x_0$.
\end{proof}

\section{Regularity and integrability of solutions}\label{sec-reg}

This section is devoted to the proof of Theorem \ref{t10novo}, which is based on some bootstrap arguments. We stress that the regularity depends on four regions of $(p, q)$-plan as illustrated by Figure \ref{figreg}. In all cases we get $W^{2,t}(\mathbb{R}^N)$-regularity for $t$ sufficient and arbitrarily large, but in some regions $t$ cannot approach one.

\begin{proof}[\bf{Proof of Theorem \ref{t10novo}}]

 Let $(u, v)$ be a solution of \eqref{P}. We divide the proof into eight steps. 
 \medbreak

\noindent{\bf Step 1.} The start of the bootstrap process: If $u\in L^s(\mathbb{R}^N)$ and $v \in L^t(\mathbb{R}^N)$ for $(s,t)$ satisfying \eqref{precisamos} and \eqref{sermaior1} ahead, then $u\in L^\sigma(\mathbb{R}^N)$ and $v \in L^\tau(\mathbb{R}^N)$ for all $(\sigma,\tau)$ in a closed rectangle. Moreover, there exists a special point $(s_0, t_0)$ satisfying \eqref{precisamos} and \eqref{sermaior1}, such that $u\in L^\sigma(\mathbb{R}^N)$ and $v \in L^\tau(\mathbb{R}^N)$ for all $(\sigma,\tau)$ in a closed rectangle having $(s_0, t_0)$ as an interior point.

Indeed, suppose that $u\in L^{s}(\mathbb{R}^N)$ and $v\in L^{t}(\mathbb{R}^N)$, for some $s>1$ and $t>1$. From \eqref{HLS2}, if
\begin{equation}\label{precisamos}
\frac{N}{\alpha}p>s>p \ \ \text{and} \ \ \frac{N}{\alpha}q>t>q,
\end{equation}
then
\begin{equation}\label{eq13}
\left\{\begin{array}{ll}
I_\alpha \ast |u|^p \in L^\gamma(\mathbb{R}^N) \mbox{ \ \ \ with \ \ \ } \dfrac{1}{\gamma} = \dfrac{p}{s}-\dfrac{\alpha}{N}, \vspace{10pt}\\
I_\alpha \ast |v|^q \in L^\beta(\mathbb{R}^N) \mbox{ \ \ \ with \ \ \ } \dfrac{1}{\beta} = \dfrac{q}{t}-\dfrac{\alpha}{N}.
\end{array} \right.
\end{equation}
Next, from $u\in L^{s}(\mathbb{R}^N)$ and $v\in L^{t}(\mathbb{R}^N)$ as in \eqref{precisamos}, we intend to find the value $r>1$ (if it exists) such that $(I_\alpha \ast |v|^q)|u|^{p-2}u \in L^{r}(\mathbb{R}^N)$. From \eqref{eq13} and Hölder inequality, we obtain the expression for $r$, namely,
\begin{equation}\label{eq14}
\dfrac{1}{r}=\dfrac{1}{s}(p-1)+\dfrac{1}{t}q-\dfrac{\alpha}{N}.
\end{equation}
Analogously, from $u\in L^{s}(\mathbb{R}^N)$ and $v\in L^{t}(\mathbb{R}^N)$ as above and \eqref{eq13}, the value $h>1$ (if it exists) such that $(I_\alpha \ast |u|^p)|v|^{q-2}v \in L^{h}(\mathbb{R}^N)$ is given by
\begin{equation}\label{eq15}
\dfrac{1}{h}=\dfrac{1}{t}(q-1)+\dfrac{1}{s}p-\dfrac{\alpha}{N}.
\end{equation}
Note that conditions $\frac{N}{\alpha}p >s$ and $\frac{N}{\alpha}q>t$ and $p,q>1$ guarantee that $r>0$ and $h>0$. Then, the conditions $r>1$ and $h>1$ read $\frac{1}{r}<1$ and $\frac{1}{h}<1$, namely
\begin{equation}\label{sermaior1}
\dfrac{1}{s}(p-1)+\dfrac{1}{t}q<\dfrac{N+\alpha}{N} \ \ \text{and} \ \ 
\dfrac{1}{t}(q-1)+\dfrac{1}{s}p<\dfrac{N+\alpha}{N}.
\end{equation}
Hence, from the classical theory of regularity for second order elliptic PDEs, if $s$ and $t$ satisfy \eqref{precisamos} and \eqref{sermaior1}, then $u\in W^{2, r}(\mathbb{R}^N)$ and $v\in W^{2, h}(\mathbb{R}^N)$, with $r$ and $h$ as in \eqref{eq14} and \eqref{eq15}, respectively. Moreover, from the Sobolev embeddings, we obtain that $u\in L^{\sigma}(\mathbb{R}^N)$ for all
\begin{equation}\label{eq19nova}
\dfrac{1}{s}(p-1)+\dfrac{1}{t}q-\dfrac{\alpha+2}{N}\leq \dfrac{1}{\sigma}\leq\dfrac{1}{s}(p-1)+\dfrac{1}{t}q-\dfrac{\alpha}{N} 
\end{equation}
and $v\in L^{\tau}(\mathbb{R}^N)$ for all
\begin{equation}\label{eq20nova}
\dfrac{1}{t}(q-1)+\dfrac{1}{s}p-\dfrac{\alpha+2}{N}\leq\dfrac{1}{\tau}\leq\dfrac{1}{t}(q-1)+\dfrac{1}{s}p-\dfrac{\alpha}{N}.
\end{equation}

Now observe that, since \eqref{H1} is assumed, let $\theta_1$ and $\theta_2$ be as in \eqref{novaeq}.  Then, $u\in L^{\theta_1 p}(\mathbb{R}^N)$ and $v\in L^{\theta_2 q}(\mathbb{R}^N)$. Set $s_0 =\theta_1 p $, $t_0 =\theta_2 q $ and observe that \eqref{precisamos} and \eqref{sermaior1} are satisfied for $s=s_0$ and $t=t_0$. Moreover, 
\[
\left[\dfrac{1}{s_0}(p-1)+\dfrac{1}{t_0}q-\dfrac{\alpha}{N}\right]  - \dfrac{1}{2} =  1 - \dfrac{1}{s_0} -\dfrac{1}{2} = \dfrac{s_0-2}{2s_0}>0 \ ,\]
\[
\left[\dfrac{1}{t_0}(q-1)+\dfrac{1}{s_0}p-\dfrac{\alpha}{N}\right]-\dfrac{1}{2}= 1 - \dfrac{1}{t_0} - \dfrac{1}{2} = \dfrac{t_0-2}{2t_0}>0 \ ,
\]
because $s_0,t_0>2$. In addition
$$
\left[ \dfrac{1}{s_0}(p-1)+\dfrac{1}{t_0}q-\dfrac{\alpha+2}{N} \right] - \frac{N-2}{2N}=  \dfrac{N-2}{2N} - \frac{1}{s_0} <0\ ,
$$
$$
\left[ \dfrac{1}{t_0}(q-1)+\dfrac{1}{s_0}p-\dfrac{\alpha+2}{N}\right] - \frac{N-2}{2N}=\dfrac{N-2}{2N} - \frac{1}{t_0} <0 \ ,
$$
since $\frac{N-2}{2N} < \frac{1}{s_0}, \frac{1}{t_0}$. Then, setting
\begin{align}
&\dfrac{1}{\underline{s}_1} = \dfrac{1}{s_0}(p-1)+\dfrac{1}{t_0}q-\dfrac{\alpha}{N}, \ \ \ \ \dfrac{1}{\underline{t}_1} = \dfrac{1}{t_0}(q-1)+\dfrac{1}{s_0}p-\dfrac{\alpha}{N}, \vspace{10pt}\label{underline}\\
& \dfrac{1}{\overline{s}_1} = \dfrac{1}{s_0}(p-1)+\dfrac{1}{t_0}q-\dfrac{\alpha+2}{N},  \ \ \ \ \dfrac{1}{\overline{t}_1} =  \dfrac{1}{t_0}(q-1)+\dfrac{1}{s_0}p-\dfrac{\alpha+2}{N} \label{overline},
\end{align}
it follows that 
\begin{equation}\label{underover}
1 > \dfrac{1}{\underline{s}_1} > \dfrac{1}{2} > \dfrac{1}{s_0}> \dfrac{N-2}{2N} > \dfrac{1}{\overline{s}_1} \ \ \text{ and } \ \ 1 > \dfrac{1}{\underline{t}_1} > \dfrac{1}{2} > \dfrac{1}{t_0}> \dfrac{N-2}{2N} > \dfrac{1}{\overline{t}_1}.
\end{equation}
Moreover, $u\in L^\sigma(\mathbb{R}^N)$ and $v \in L^\tau(\mathbb{R}^N)$ for all $\sigma, \tau\geq 1$ such that
\begin{equation}\label{firstrange}
\dfrac{1}{\underline{s}_1} \geq \dfrac{1}{\sigma} \geq \dfrac{1}{\overline{s}_1} \ \ \text{ and } \ \ \dfrac{1}{\underline{t}_1} \geq \dfrac{1}{\tau} \geq \dfrac{1}{\overline{t}_1}.
\end{equation}

\medbreak

For the readers convenience, we illustrate at Figures \ref{pichipa}-\ref{pichipd} the points $(s,t)$ as in \eqref{precisamos} and \eqref{sermaior1} with $(p,q)$ as in Theorem \ref{t10novo} (a)-(d). These pictures are also intended to illustrate the next steps on the bootstrap argument.

\begin{figure}[!htb]\centering
   \begin{minipage}{0.45\textwidth}
     \includegraphics[width=.8\linewidth]{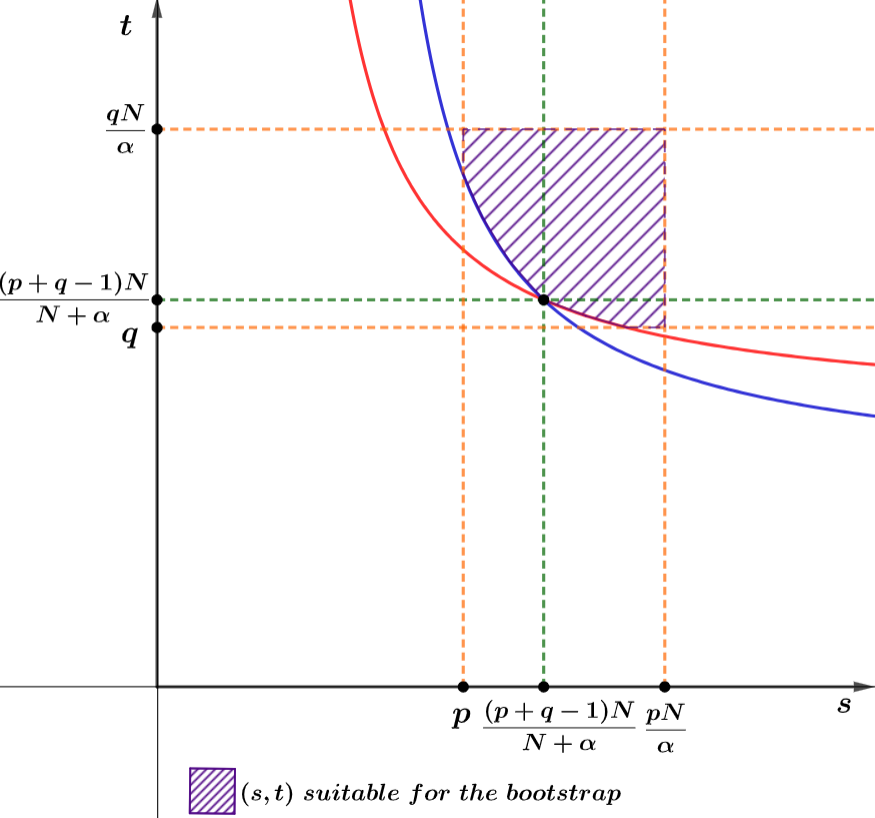}
     \caption{Points $(s,t)$ as in \eqref{precisamos} and \eqref{sermaior1} with $(p,q)$ as in Theorem \ref{t10novo} (a).}\label{pichipa}
   \end{minipage}
   \hspace{1cm}
   \begin {minipage}{0.45\textwidth}
     \includegraphics[width=.8\linewidth]{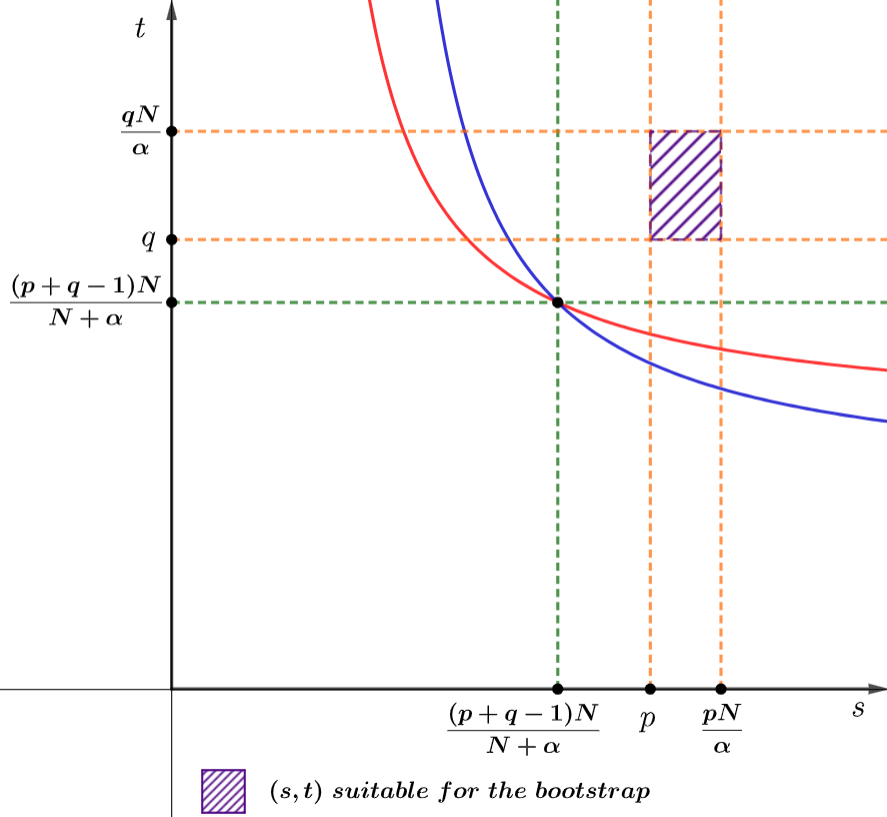}
     \caption{Points $(s,t)$ as in \eqref{precisamos} and \eqref{sermaior1} with $(p,q)$ as in Theorem \ref{t10novo} (b).}\label{pichipb}
   \end{minipage}
\end{figure}

\begin{figure}[!htb]\centering
   \begin{minipage}{0.45\textwidth}
     \includegraphics[width=.8\linewidth]{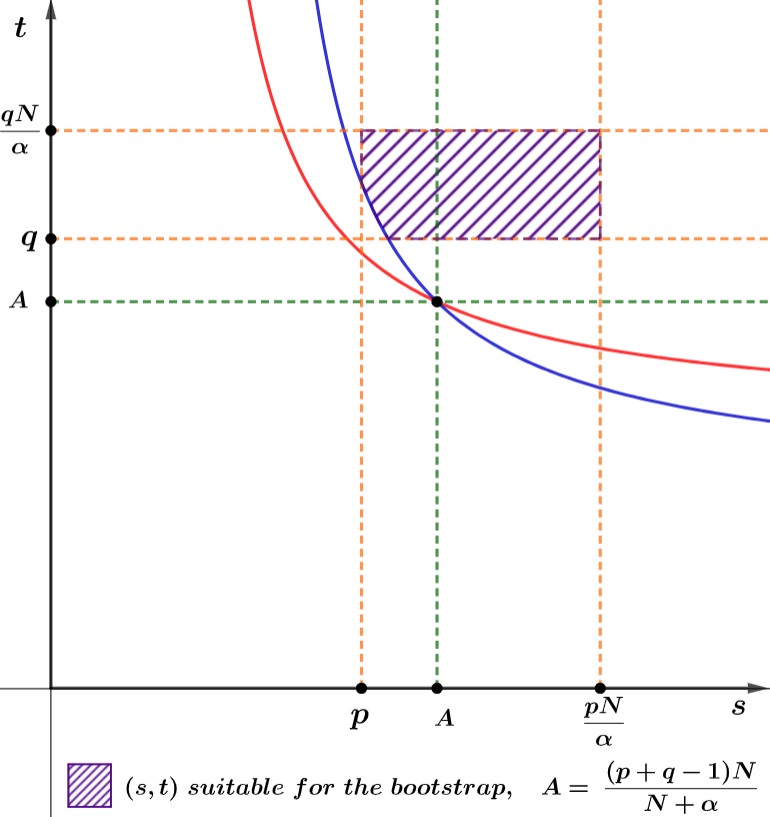}
     \caption{Points $(s,t)$ as in \eqref{precisamos} and \eqref{sermaior1} with $(p,q)$ as in Theorem \ref{t10novo} (c).}\label{pichipc}
   \end{minipage}
   \hspace{1cm}
   \begin {minipage}{0.45\textwidth}
     \includegraphics[width=.8\linewidth]{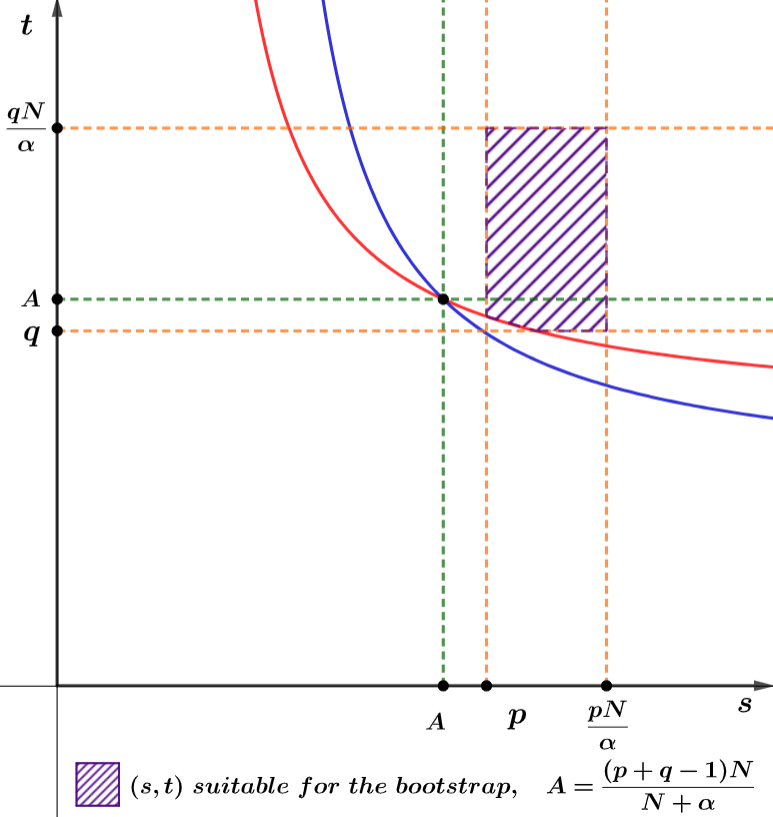}
     \caption{Points $(s,t)$ as in \eqref{precisamos} and \eqref{sermaior1} with $(p,q)$ as in Theorem \ref{t10novo} (d).}\label{pichipd}
   \end{minipage}
\end{figure}

\medbreak

\newpage
\noindent {\bf Step 2.} There exists $1  < \underline{b} <2$, and for $N\geq3$ there exists $\overline{b} > \frac{2N}{N-2}$ such that $u, v \in L^{b}(\mathbb{R}^N)$ 
\[
\forall \, b\in [\underline{b}, \infty), \ \ \text{for} \ \ N=1,2; \ \ \ \forall \, b\in [\underline{b}, \overline{b}], \ \ \text{for} \ \ N\geq3.
\]

Indeed, from \eqref{underover} and \eqref{firstrange}, take $\underline{b} = \max\{\underline{s}_1, \underline{t}_1\}$ and, for $N \geq 3$, choose any $\overline{b}>0$ such that
\[
\max\left\{ \dfrac{1}{\overline{s}_1}, \dfrac{1}{\overline{t}_1}\right\} \leq \dfrac{1}{\overline{b}} < \dfrac{N-2}{2N}.
\]

\medbreak

\noindent {\bf Step 3.} There exist $1  < \underline{a} <2 < \min\{s_0, t_0\} \leq \max\{s_0, t_0\}< \overline{a} $, such that $u, v \in L^{a}(\mathbb{R}^N)$ for all $a\in [\underline{a}, \overline{a}]$ and such that
\begin{equation}\label{bbarra}
\dfrac{p+q-2}{\overline{a}} < \frac{\alpha+2}{N} \ \ \text{and }\ \ \dfrac{N}{2}\max\{(q-2), (p-2)\} < \overline{a}. 
\end{equation}

Indeed, from Step 2, we can take $\underline{a} = \underline{b}$. For $N=1,2$, from Step 2, on can choose $\overline{a}$ sufficiently large such that \eqref{bbarra} holds. For $N\geq 3$, take $\overline{a} = \overline{b}$ from Step 2, and observe that
\[
\dfrac{p+q-2}{\overline{a}} < \dfrac{(p+q-2)(N-2)}{2N} < \dfrac{N+\alpha}{N} - \dfrac{N-2}{N}= \dfrac{\alpha+2}{N},
\]
where the last inequality comes from \eqref{H1}, namely $p+q< \frac{2(N+\alpha)}{N-2}$. Moreover, for $N\geq 3$, since $p,q < \frac{2N}{N-2}$
\[
\dfrac{N}{2}\max\{(q-2), (p-2)\} < \dfrac{2N}{N-2} < \overline{a}.
\]

\medbreak 

\noindent {\bf Step 4.} Suppose that $(s,t)$ satisfies \eqref{precisamos} and \eqref{sermaior1}, and consider $r$ and $h$ given by \eqref{eq14} and \eqref{eq15}, respectively. Then
\begin{equation}\label{alternancia}
s \leq t \Longrightarrow r\geq h \ \ \ \text{and} \ \ s \geq t \Longrightarrow r \leq h.
\end{equation}

Indeed, from \eqref{eq14} and \eqref{eq15}, it follows that
\[
\dfrac{1}{r} + \dfrac{1}{s} = \dfrac{1}{h} + \dfrac{1}{t},
\]
which implies \eqref{alternancia}.

\medbreak

\noindent {\bf Step 5.} For all $r, h>1$ verifying
 $$
\dfrac{1}{\overline{r}}> \dfrac{1}{r}> \dfrac{\alpha}{N}\left(1-\dfrac{1}{p}\right)  \mbox{ \ \ \ and \ \ \ } \dfrac{1}{\overline{h}}> \dfrac{1}{h}> \dfrac{\alpha}{N}\left(1-\dfrac{1}{q}\right),
 $$
it holds that $u\in W^{2, r}(\mathbb{R}^N)$ and $v\in W^{2, h}(\mathbb{R}^N)$. Then, from the Sobolev embeddings, for all $s, t>1$ satisfying
 $$
\dfrac{1}{\overline{r}} > \dfrac{1}{s}> \dfrac{\alpha}{N}\left(1-\dfrac{1}{p}\right)-\dfrac{2}{N} \mbox{ \ \ \ and \ \ \ }  \dfrac{1}{\overline{h}}> \dfrac{1}{t}> \dfrac{\alpha}{N}\left(1-\dfrac{1}{q}\right)-\dfrac{2}{N},
 $$
 it holds that $u\in L^{s}(\mathbb{R}^N)$ and $v\in L^{t}(\mathbb{R}^N)$.  Here
 \[
\begin{cases}
\overline{r} =\overline{h} = 1 \quad \text{at case (a),}\\
\overline{r}= r^{\ast}, \ \ \overline{h}= h^{\ast} \quad\text{at case (b),}\\
\overline{r}= r^{\ast \ast}, \ \ \overline{h}= h^{\ast \ast} \quad\text{at case (c),}\\ 
\overline{r}=  r^{\ast \ast \ast}, \ \ \overline{h}=  h^{\ast \ast \ast} \quad\text{at case (d).}
\end{cases}
\]

\noindent {\bf The descending process.} Let $\underline{a}$ from Step 3. We split the argument into two cases. 
\medbreak

\noindent {\bf Case 1.} The point $(\underline{a}, \underline{a})$ satisfies \eqref{sermaior1} and $\underline{a}> p,q$.

In this case, since $\underline{a} < s_0, t_0$, it follows that $(\underline{a}, \underline{a})$ also satisfies \eqref{precisamos}.  Then, return to Step 1, with a point $(s, s)$, with $s\leq \underline{a}$ satisfying \eqref{precisamos} and \eqref{sermaior1}. Then, from \eqref{eq14}, \eqref{eq15} and \eqref{H1}, $r= h$, $u, v \in W^{2, r}(\mathbb{R}^N)$ and 

\begin{equation}\label{decrescente}
\dfrac{1}{r}-\dfrac{1}{s}= \dfrac{1}{h} - \dfrac{1}{s}= \dfrac{p+q-2}{s}-\dfrac{\alpha}{N}\geq \dfrac{p+q-2}{\underline{a}}-\dfrac{\alpha}{N} > \dfrac{p+q-2}{2}-\dfrac{\alpha}{N}>  0.
\end{equation}

This means that $r=h<s$, and there is a uniform gap between $r$ and $s$. Next we analyse this descent process in each region stated in the theorem.

\medbreak

\noindent {\bf Region (a).} From \eqref{decrescente}, the uniform gap, after a finite number of steps, it follows that $u,v \in L^{s}(\mathbb{R}^N)$ for all 
$$\frac{(p+q-1)N}{N+\alpha} < s \leq \underline{a}$$
and for all $s$ in this range, $(s,s)$ satisfies  \eqref{precisamos} and \eqref{sermaior1}. Finally, plugging \linebreak $(s,t) = \left( \frac{(p+q-1)N}{N+\alpha} +\varepsilon,\frac{(p+q-1)N}{N+\alpha} +\varepsilon \right)$ at \eqref{eq14} and \eqref{eq15}, it follows $r=h\to 1$ as $\varepsilon \to 0$, which implies  $\overline{r} = \overline{h}=1$.

\medbreak

\noindent {\bf Region (b).} Suppose, without loss of generality, that $q\geq p$. The objective is to descend as close as possible to the point $(p,q)$, and then give the final step. From \eqref{decrescente}, the uniform gap, after a finite number of steps, it follows that $u,v \in L^{s}(\mathbb{R}^N)$, for all 
\[
q < s \leq \underline{a}.
\]

If $q=p$, the process has already reached the limit point $(p,q)$. 

If $q>p$, then taking the points of $(s,t) = (q+\varepsilon, q +\varepsilon)$, from the uniform gap \eqref{decrescente}, it follows that $u\in L^s(\mathbb{R}^N)$, $v \in L^t(\mathbb{R}^N)$ with (indeed the range would be even better)
\[
q \leq s \leq \underline{a} \ \ \text{ and } \ \ q < t \leq \underline{a},
\]
and all these pairs $(s,t)$ satisfy \eqref{precisamos} and \eqref{sermaior1}. Recall that we are treating the case $p< q < \underline{a}<2$.

If $\frac{2N}{2N-\alpha} \leq p < q$. Then, for all $p < s_0 < q$, and for all $(s,t) = (s,q)$ with $s_0 \leq s \leq q$ (indeed one must take $t=q+ \varepsilon$), it follows that
\begin{align}
\dfrac{1}{r} - \frac{1}{s} & = \dfrac{p-2}{s} +\dfrac{q}{q} - \dfrac{\alpha}{N} \geq \dfrac{p-2}{s_0} + \dfrac{N - \alpha}{N} \label{ateaqui} \\
& > \dfrac{p-2}{p} + \dfrac{N-\alpha}{N} > \dfrac{p-2}{p} + \dfrac{N+\alpha}{N} = \dfrac{p(2N-\alpha)-2N}{pN}\geq 0. \nonumber
\end{align}
From this uniform gap, after a finite number of steps, we infer that $u\in L^s(\mathbb{R}^N)$, $v \in L^t(\mathbb{R}^N)$ for all
\[
s_0 <s < \underline{a} \ \ \text{ and } \ \ q < t \leq \underline{a}.
\]
Since $p < s_0< q$, was arbitrary, we conclude that $u\in L^s(\mathbb{R}^N)$, $v \in L^t(\mathbb{R}^N)$ for all
\[
p <s < \underline{a} \ \ \text{ and } \ \ q < t \leq \underline{a}.
\]

One could argue similarly in the case $\frac{2N}{2N -\alpha} \leq q \leq p$. Therefore, if $p,q \geq \frac{2N}{2N-\alpha}$, then $u\in L^s(\mathbb{R}^N)$, $v \in L^t(\mathbb{R}^N)$ for all
\[
p <s < \underline{a} \ \ \text{ and } \ \ q < t \leq \underline{a}.
\]
Finally, plugging $(s,t) = \left( p +\varepsilon, q +\varepsilon \right)$ at \eqref{eq14} and \eqref{eq15}, it follows that
\[
\dfrac{1}{r} \to \dfrac{p-1}{p} + \dfrac{N-\alpha}{N} = \dfrac{(2N-\alpha)p -N}{pN}, \ \ \ \dfrac{1}{h} \to \dfrac{q-1}{q} + \dfrac{N-\alpha}{N} = \dfrac{(2N-\alpha)q -N}{qN} \ \ \text{as $\varepsilon \to 0$.}
\]

Now, consider the case $p \leq \frac{2N}{2N - \alpha}$ (then, by \eqref{H1}, $\frac{2N}{2N - \alpha} < q$) and set $s_0 = \frac{(2-p)N}{N-\alpha}$ (observe that $p <s_0$). So, arguing as in \eqref{ateaqui}, we infer that $u\in L^s(\mathbb{R}^N)$, $v \in L^t(\mathbb{R}^N)$ for all
\[
\dfrac{(2-p)N}{N-\alpha} <s < \underline{a} \ \ \text{ and } \ \ q < t \leq \underline{a}.
\]
Then, plugging $(s,t) = \left( \frac{(2-p)N}{N-\alpha} +\varepsilon, q +\varepsilon \right)$ at \eqref{eq14} and \eqref{eq15}, it follows that
\begin{align*}
&\dfrac{1}{r} \to \dfrac{(p-1)(N-\alpha)}{(2-p)N} + \dfrac{N-\alpha}{N} = \dfrac{N-\alpha}{(2-p)N}, \\ & \dfrac{1}{h} \to \dfrac{q-1}{q} + \dfrac{p(N -\alpha)}{(2-p)N} - \dfrac{\alpha}{N} = \dfrac{(2q+p-2)N - 2q\alpha}{(2-p)qN} \ \ \text{as $\varepsilon \to 0$.}
\end{align*}

Similarly, if $q \leq \frac{2N}{2N - \alpha}$ (then, by \eqref{H1}, $\frac{2N}{2N - \alpha} < p$ ), by plugging $(s,t) = \left( p+ \varepsilon , \frac{(2-q)N}{N-\alpha} +\varepsilon\right)$ at \eqref{eq14} and \eqref{eq15}, it follows that
\[
\dfrac{1}{r} \to \dfrac{(2p+q-2)N - 2p\alpha}{(2-q)pN} , \ \ \  \dfrac{1}{h} \to   \dfrac{N-\alpha}{(2-q)N} \ \ \text{as $\varepsilon \to 0$.}
\]

\noindent {\bf Region (c).} In this case, the objective is to make $(s,t)$ as close as possible to the intersection point of the line $t=q$ and the hyperbola $
\frac{1}{t}(q-1)+\frac{1}{s}p=\frac{N+\alpha}{N}$, given by $\left(\frac{pqN}{N+\alpha q}, q\right)$. 

Observe that the conditions describing Region (c) implies that
\[
p \leq \frac{pqN}{N+\alpha q}<q.
\]
So, as in the case for the Region (b), using the uniform gap \eqref{decrescente}, after a finite number of steps, we infer that $u, v \in L^s(\mathbb{R}^N)$ for all
\[
q < s \leq \underline{a}.
\]
Hence, as in the case for Region (b), if $\frac{(2-p)N}{N-\alpha} \leq \frac{pqN}{N + \alpha q}$, then $u \in L^s(\mathbb{R}^N), v \in L^{t}(\mathbb{R}^N)$ for all
\[
\dfrac{pqN}{N + \alpha q} < s \leq \underline{a} \ \ \text{and} \ \ q< t \leq \underline{a}.
\]
By plugging $(s,t) = \left( \frac{pqN}{N + \alpha q} +\varepsilon, q +\varepsilon \right)$ at \eqref{eq14} and \eqref{eq15}, it follows that
\begin{align*}
&\dfrac{1}{r} \to \dfrac{(p-1)(N +\alpha q)}{pqN} + \dfrac{N-\alpha}{N} = \dfrac{(pq+p-1)N -q \alpha}{pqN}, \\
& \dfrac{1}{h} \to \dfrac{q-1}{q} + \dfrac{p(N + \alpha q)}{pqN} - \dfrac{\alpha}{N} = 1 \ \ \text{as} \ \ \varepsilon \to 0.
\end{align*}

Also, as in the case for Region (b), if $\frac{(2-p)N}{N-\alpha} > \frac{pqN}{N + \alpha q}$, then $u \in L^s(\mathbb{R}^N), v \in L^{t}(\mathbb{R}^N)$ for all
\[
\dfrac{(2-p)N}{N - \alpha} < s \leq \underline{a} \ \ \text{and} \ \ q< t \leq \underline{a}.
\]
Then, plugging $(s,t) = \left( \frac{(2-p)N}{N-\alpha} +\varepsilon, q +\varepsilon \right)$ at \eqref{eq14} and \eqref{eq15}, it follows that
\begin{align*}
&\dfrac{1}{r} \to \dfrac{(p-1)(N-\alpha)}{(2-p)N} + \dfrac{N-\alpha}{N} = \dfrac{N-\alpha}{(2-p)N}, \\ & \dfrac{1}{h} \to \dfrac{q-1}{q} + \dfrac{p(N -\alpha)}{(2-p)N} - \dfrac{\alpha}{N} = \dfrac{(2q+p-2)N - 2q\alpha}{(2-p)qN} \ \ \text{as $\varepsilon \to 0$.}
\end{align*}

\medbreak

\noindent {\bf Region (d).} For this region, one can apply, mutatis mutandis, the arguments from Region (c).

\medbreak

\noindent {\bf Case 2.} The point $(\underline{a}, \underline{a})$ does not satisfy \eqref{sermaior1} or $\underline{a}\leq p$ or $\underline{a} \leq q$.

Observe that,  from \eqref{underover} and \eqref{firstrange}, Steps 2 and 3, we infer that $u,v\in L^s(\mathbb{R}^N)$ for all $\underline{a} \leq s  \leq \overline{a}$.

\medbreak
\noindent {\bf Region (a).} Observe that, in this case, $(\underline{a}, \underline{a})$ does not satisfy \eqref{sermaior1}, and then $\underline{a} \leq \frac{(p+q-1)N}{N+\alpha}< \overline{a}$. Then, since  $u,v\in L^s(\mathbb{R}^N)$ for all $\underline{a} \leq s  \leq \overline{a}$, as in the previous case, plugging \linebreak $(s,t) = \left( \frac{(p+q-1)N}{N+\alpha} +\varepsilon,\frac{(p+q-1)N}{N+\alpha} +\varepsilon \right)$ at \eqref{eq14} and \eqref{eq15}, it follows $r=h\to 1$ as $\varepsilon \to 0$, which implies  $\overline{r} = \overline{h}=1$.

\medbreak

\noindent {\bf Region (b).} Suppose, without loss of generality, $q\geq p$. Hence, in this case $\underline{a} \leq q$. Then, since $u,v\in L^s(\mathbb{R}^N)$ for all $\underline{a} \leq s  \leq \overline{a}$, one can argue as in the proof of the previous case for the Region (b), from the point that $u,v\in L^s(\mathbb{R}^N)$ for all $q \leq s  \leq \overline{a}$.

\medbreak
\noindent {\bf Region (c).} As written in the previous case, in this region
\[
p \leq \frac{pqN}{N+\alpha q}<q.
\]
and, the condition in Case 2 says that $q\geq \underline{a}$. Then, since $u,v\in L^s(\mathbb{R}^N)$ for all $\underline{a} \leq s  \leq \overline{a}$, one can argue as in the proof of the previous case for the Region (c), from the point that $u,v\in L^s(\mathbb{R}^N)$ for all $q \leq s  \leq \overline{a}$. 

\medbreak

\noindent {\bf Region (d).} For this region, one can apply, mutatis mutandis, the arguments from Region (c) from the last paragraph.

\medbreak
\noindent {\bf The ascending process.}

Given $(s,t)$ satisfying \eqref{precisamos} and \eqref{sermaior1}, from \eqref{eq19nova} and \eqref{eq20nova}, set
\begin{equation}\label{stoverline}
\dfrac{1}{\overline{s}} = \dfrac{p-1}{s} + \dfrac{q}{t} - \dfrac{\alpha+2}{N} \ \ \text{ and } \ \ \dfrac{1}{\overline{t}} = \dfrac{q-1}{t} + \dfrac{p}{s} - \dfrac{\alpha+2}{N}.
\end{equation}

Consider $\overline{a}$ from Step 3. The first common target, in all the four regions, is to prove that $u \in L^s(\mathbb{R}^N), v \in L^t(\mathbb{R}^N)$ for all $(s,t) \in [\overline{a}, \frac{pN}{\alpha}) \times [\overline{a}, \frac{qN}{\alpha})$.

\medbreak

\noindent{\bf Case 1. } The point $\overline{a}$ from Step 3 satisfies $\overline{a} < \frac{N}{\alpha} \min\{p,q\}$.

In this case, since $\max\{s_0, t_0\} < \overline{a}$, the point $(\overline{a}, \overline{a})$ also satisfies \eqref{precisamos} and \eqref{sermaior1}. In addition, for all $s=t\geq \overline{a}$, from \eqref{eq19nova} and \eqref{eq20nova}, it follows that $u, v \in L^{\sigma}(\mathbb{R}^N)$ for all $\sigma \geq \overline{a}$ such that
\[
\dfrac{p+q-1}{s}- \dfrac{\alpha+2}{N} \leq \dfrac{1}{\sigma} \leq \frac{1}{\overline{a}} 
\]
and, by \eqref{bbarra}, 
\begin{equation}\label{crescente}
\dfrac{1}{\overline{s}} - \dfrac{1}{s} = \left[\dfrac{p+q-1}{s}- \dfrac{\alpha+2}{N}\right] -\dfrac{1}{s}= \dfrac{p+q-2}{s} - \dfrac{\alpha+2}{N} \leq \dfrac{p+q-2}{\overline{a}}-\dfrac{\alpha+2}{N}<0,
\end{equation}
which gives a uniform gap for the ascending integrability of $u$ and $v$.

Without loss of generality, suppose that $q\geq p$. Then, from the uniform gap \eqref{crescente}, after a finite number of steps, we infer that $u, v \in L^{\sigma}(\mathbb{R}^N)$ for all $\sigma$ such that
\[
\overline{a} \leq \sigma \leq \frac{pN}{\alpha} + \delta,
\]
for some $\delta>0$. 

If $q=p$, then one has reached the first target, namely $u \in L^s(\mathbb{R}^N), v \in L^t(\mathbb{R}^N)$ for all $(s,t) \in [\overline{a}, \frac{pN}{\alpha}) \times [\overline{a}, \frac{qN}{\alpha})$.

If $q>p$, then observe that for all $\overline{a}<\frac{pN}{\alpha} \leq t < \frac{qN}{\alpha}$ and $s = \frac{pN}{\alpha} - \varepsilon$, it follows that $(s,t)$ satisfies \eqref{precisamos} and \eqref{sermaior1}. Moreover, from \eqref{stoverline} (at the limit as $\varepsilon \to 0$)
\[
\dfrac{1}{\overline{t}} - \frac{1}{t} \to \dfrac{q-2}{t} + \frac{p\alpha}{pN} - \dfrac{\alpha+ 2}{N} = \dfrac{q-2}{t} - \dfrac{2}{N} <0,
\]
which follows from $t \geq \overline{a}$ and the second inequality at \eqref{bbarra}. This again produces a uniform gap between, $\overline{t}$ and $t$. Then after a finite number of steps, we infer that $u \in L^s(\mathbb{R}^N), v \in L^t(\mathbb{R}^N)$ for all $(s,t) \in [\overline{a}, \frac{pN}{\alpha}) \times [\overline{a}, \frac{qN}{\alpha})$. This finishes the first target of this ascending process. 

Finally, taking $(s,t) = (\frac{pN}{\alpha} - \varepsilon, \frac{qN}{\alpha} - \varepsilon)$, which satisfies \eqref{precisamos} and \eqref{sermaior1} (for $\varepsilon>0$ small), we infer from \eqref{eq14} that
\[
\dfrac{1}{r} \to \dfrac{(p-1)\alpha}{pN} + \dfrac{q\alpha}{qN} - \dfrac{\alpha}{N} = \dfrac{\alpha}{N}\left( 1 - \dfrac{1}{p}\right)
\]
and similarly, from \eqref{eq15}, that
\[
\dfrac{1}{h} \to \dfrac{\alpha}{N}\left( 1 - \dfrac{1}{q}\right).
\]

This finishes the proof of the ascending process in this case.
\medbreak

\noindent{\bf Case 2. } The point $\overline{a}$ from Step 3 satisfies $\overline{a} \geq \frac{N}{\alpha} \min\{p,q\}$.

If $\overline{a}$ also satisfies $\overline{a} \geq \frac{N}{\alpha} \max\{p,q\}$, then $u \in L^s(\mathbb{R}^N), v \in L^t(\mathbb{R}^N)$ for all $(s,t) \in [\overline{a}, \frac{pN}{\alpha}) \times [\overline{a}, \frac{qN}{\alpha})$, and the first target of the ascending process is completed. Then, one proceed as in the previous case by taking $(s,t) = (\frac{pN}{\alpha} - \varepsilon, \frac{qN}{\alpha} - \varepsilon)$ to complete the proof of this step.

So, consider the case
\[
\dfrac{N}{\alpha} \min\{p,q\} \leq \overline{a} <\dfrac{N}{\alpha} \max\{p,q\}
\]

Again, without loss of generality, suppose that $q\geq p$.

Then, from Step 3, $u \in L^s(\mathbb{R}^N)$, $v \in L^t(\mathbb{R}^N)$ for all $\underline{a} \leq s \leq \frac{pN}{\alpha}$ and $\underline{a} \leq t \leq \overline{a}$. Then observe that the points of the form $(s,t) = (\frac{pN}{\alpha}- \varepsilon, t)$ with $\overline{a} \leq t < \frac{qN}{\alpha}$ satisfies \eqref{precisamos} and \eqref{sermaior1}, and from \eqref{stoverline}
\[
\dfrac{1}{\overline{t}} - \dfrac{1}{t} \to \frac{q-2}{t} - \frac{2}{N} < 0,
\]
because, as proved in Step 3, $\frac{N}{2}(q-2)< \overline{a}$ and $t\geq \overline{a}$.  This again produces a uniform gap between, $\overline{t}$ and $t$. Then after a finite number of steps, we infer that $u \in L^s(\mathbb{R}^N), v \in L^t(\mathbb{R}^N)$ for all $(s,t) \in [\overline{a}, \frac{pN}{\alpha}) \times [\overline{a}, \frac{qN}{\alpha})$. This finishes the first target of this ascending process.

Then, from this point, we may argue as in the previous case to conclude the proofs of the ascending process and of Step 5.

\medbreak
\noindent {\bf Step 6.} It holds that $I_\alpha \ast |u|^p \in L^{\infty}(\mathbb{R}^N)$ and $I_\alpha \ast |v|^q \in L^{\infty}(\mathbb{R}^N)$. 

From $\alpha\in (0, N)$ and $(N-2)p< 2N$, it follows that $(\alpha-2)p<2\alpha$. This implies that $\frac{\alpha}{Np}>\frac{\alpha}{N}\left(1-\frac{1}{p}\right)-\frac{2}{N}$. Analogously, $\frac{\alpha}{Nq}>\frac{\alpha}{N}\left(1-\frac{1}{q}\right)-\frac{2}{N}$. Then, from Step 3, there exists $\varepsilon > 0$ such that $u\in L^{s}(\mathbb{R}^N)$ and $v\in L^{t}(\mathbb{R}^N)$, for all $s\in \left[\frac{Np}{\alpha}-\varepsilon, \frac{Np}{\alpha}+\varepsilon\right]$ and $t\in \left[\frac{Nq}{\alpha}-\varepsilon, \frac{Nq}{\alpha}+\varepsilon\right]$, respectively. 

Therefore, for $\delta >0$ suitably small, for all $x \in \mathbb{R}^N$, by Hölder inequality,
\begin{align*}
 |I_\alpha \ast |u|^p(x)| & = C\left( \int_{B_1(x)} \dfrac{1}{|x-y|^{N-\alpha}}|u(y)|^p dy + \int_{B_{1}^{c}(x)} \dfrac{1}{|x-y|^{N-\alpha}}|u(y)|^p dy\right)\\
 & \leq C\left(\int_{B_1(x)} \dfrac{1}{|x-y|^{N-\delta}} dy\right)^{\frac{N-\alpha}{N-\delta}}\left(\int_{B_1(x)} |u(y)|^{p\frac{N-\delta}{\alpha - \delta}} dy\right)^{\frac{\alpha-\delta}{N-\delta}}+ \\
 & + C\left(\int_{B_{1}^{c}(x)} \dfrac{1}{|x-y|^{N+\delta}} dy\right)^{\frac{N-\alpha}{N+\delta}}\left(\int_{B_{1}^{c}(x)} |u(y)|^{p\frac{N+\delta}{\alpha +\delta}} dy\right)^{\frac{\alpha+\delta}{N+\delta}}\\
 & \leq C_1 \left(\int_{\mathbb{R}^N} |u|^{p\frac{N-\alpha}{N-\delta}}dy \right)^{\frac{N-\alpha}{N-\delta}}+C_2 \left(\int_{\mathbb{R}^N} |u|^{p\frac{N+\delta}{\alpha +\delta}}dy\right)^{\frac{\alpha-\delta}{N-\delta}}.
\end{align*}
The kernel $I_\alpha \ast |v|^q $ is treated in a similar way.

\medbreak
\noindent {\bf Step 7.} It holds that $u\in W^{2, r}(\mathbb{R}^N)$ and $v\in W^{2, h}(\mathbb{R}^N)$, for all $r \in (\overline{r}, +\infty)$ and $h\in (\overline{h}, +\infty)$, with $\overline{r}$ and $\overline{h}$ as in Step 5 (as in cases (a)-(d) from the statement of this theorem).

Set $r_0, h_0\in (1, +\infty)$ by $\frac{1}{r_0}=\frac{\alpha}{N}\left(1-\frac{1}{p}\right)$ and $\frac{1}{h_0}=\frac{\alpha}{N}\left(1-\frac{1}{q}\right)$, respectively. Then, from Step 5, it follows that $u\in W^{2, r}(\mathbb{R}^N)$ for all $r\in (\overline{r}, r_0)$, and $v\in W^{2, h}(\mathbb{R}^N)$ for all $h\in (\overline{h}, h_0)$. 

Suppose that $u\in W^{2, r}(\mathbb{R}^N)$, for all $r\in (\overline{r}, r_n)$, and $v\in W^{2, h}(\mathbb{R}^N)$, for all $h\in (\overline{h}, h_n)$. Thus, from the Sobolev embeddings, $u\in L^{s}(\mathbb{R}^N)$ for all $s>\overline{r}$ such that $\frac{1}{s}>\frac{1}{r_n}-\frac{2}{N}$ and $v\in L^{t}(\mathbb{R}^N)$ for all $t>\overline{h}$ satisfying $\frac{1}{t}>\frac{1}{h_n}-\frac{2}{N}$. 

Hence, from Step 6, $(I_\alpha \ast |v|^q)|u|^{p-2}u \in L^{r}(\mathbb{R}^N)$ for every $r>1$ such that $\frac{(p-1)}{\overline{r}}> \frac{1}{r}>(p-1)\left(\frac{1}{r_n}-\frac{2}{N}\right) $ and $(I_\alpha \ast |u|^p)|v|^{q-2}v \in L^{h}(\mathbb{R}^N)$ for every $h>1$ satisfying  $\frac{(q-1)}{\overline{h}}> \frac{1}{h}>(q-1)\left(\frac{1}{h_n}-\frac{2}{N}\right) $. By the classical Calderón-Zygmund theory, $u\in W^{2,r}(\mathbb{R}^N)$ and $v\in W^{2,h}(\mathbb{R}^N)$ for every $r>\overline{r}$ such that $\frac{1}{r}>(p-1)\left(\frac{1}{r_n}-\frac{2}{N}\right)$ and for every $h>\overline{h}$ such that $\frac{1}{h}>(q-1)\left(\frac{1}{h_n}-\frac{2}{N}\right)$, respectively. 

Now, observe that, if $r_n \geq \frac{N}{2}$ nothing remains to be proved regarding the $W^{2,r}(\mathbb{R}^N)$-regularity of $u$. Otherwise, define
\begin{equation}\label{eqrn}
\dfrac{1}{r_{n+1}}= (p-1)\left(\frac{1}{r_n}-\frac{2}{N}\right).
\end{equation}
If $p\leq 2$, then $(r_n)$ is an increasing sequence and, after a finite number of steps, it follows that $r_n \geq \frac{N}{2}$. So, suppose $p>2$. Then observe that one can take $\frac{1}{r_n}<\frac{\alpha}{N}\left(1-\frac{1}{p}\right)$. Indeed, from \eqref{eqrn}, one can take
\[
\frac{1}{r_n} = (p-1)\left(\frac{1}{r}-\frac{2}{N}\right), \quad \text{with $r < r_0$ and $r\approx r_0$ }
\]
and
\[
(p-1)\left(\frac{1}{r_0}- \frac{2}{N}\right) < \frac{\alpha}{N}\left( 1 - \frac{1}{p} \right)
\]
because $(\alpha-2)p < 2\alpha$, as explained at the beginning of Step 6. Then, since $\frac{1}{r_n}<\frac{\alpha}{N}\left(1-\frac{1}{p}\right)$ and $\frac{1}{p}\geq \frac{N-2}{2N}$, we infer that
\begin{align*}
    \dfrac{1}{r_{n+1}} &< \dfrac{1}{r_n}+(p-1)\left[\dfrac{\alpha}{N}\left(1-\dfrac{2}{p}\right)-\dfrac{2}{N}\right]< \dfrac{1}{r_n}+(p-1)\left[\dfrac{\alpha}{N}\left(1-2\dfrac{N-2}{2N}\right)-\dfrac{2}{N}\right]\\
    &< \dfrac{1}{r_n}+(p-1)\dfrac{2}{N}\left(\dfrac{\alpha}{N}-1\right)< \dfrac{1}{r_n}.
\end{align*}
Therefore, after a finite number of steps, one reaches the conclusion. The procedure for $(h_n)$ is similar, and hence omitted.

\medbreak
\noindent{\bf Step 8.} It holds that $u, v\in L^{\infty}(\mathbb{R}^N) \cap C^2(\mathbb{R}^N)$ and $u \in C^{\infty}(\mathbb{R}^N \setminus u^{-1}(\{0\}))$ and $v \in C^{\infty}(\mathbb{R}^N \setminus v^{-1}(\{0\}))$.

First, from the Sobolev embedding $W^{2,\tau}(\mathbb{R}^N) \hookrightarrow L^{\infty}(\mathbb{R}^N)$ for $\tau> N/2$ and Step 7, we infer that $u,v \in L^{\infty}(\mathbb{R}^N)$. Then, by classical local elliptic regularity, as in the proof of \cite[Proposition 4.1]{Moroz2013}, we infer that $u \in C^2(\mathbb{R}^N)\cap C^{\infty}(\mathbb{R}^N \setminus u^{-1}(\{0\}))$ and $v \in C^2(\mathbb{R}^N)\cap C^{\infty}(\mathbb{R}^N \setminus v^{-1}(\{0\}))$.
\end{proof}

\section{Decay at infinity for ground state solutions}\label{sec-decay}

We begin this section by discussing on how the regularity of solutions, given by Theorem \ref{t10novo}, contribute for their decay at infinity. When studying the decay at infinity for $u$, the integrability of $v$ will take a central part, and vice versa. More precisely, to study the decay for $u$, one should observe that the decay of the kernel $I_\alpha \ast |v|^q$  strongly interfere in the arguments, making necessary to obtain estimates for it, which is given at Lemma \ref{decaimento}.

\subsection{Proof of Theorem \ref{decaimentosol} for the cases $p>2$ or $q>2$.}\label{duasquad}

In this part we consider the case when either $p>2$ or $q>2$. For the readers convenience, we start recalling \cite[Lemma 6.4]{Moroz2013}. 

\begin{lemma}\label{lpm2}
Let $\rho \geq 0$ and $W\in C^1((\rho, +\infty), \mathbb{R})$. If $\lim\limits_{s\rightarrow +\infty}W(s) > 0$ and for some $\beta>0$, $\lim\limits_{s\rightarrow + \infty}W'(s)s^{1+\beta} = 0$,
then there exists a nonnegative radial function $v:\mathbb{R}^N \setminus B_\rho \rightarrow \mathbb{R}$ such that
$$
-\Delta v + W v = 0 \ \ \text{in} \ \ \mathbb{R}^N \setminus B_\rho,
$$
and for some $\rho_0 \in (\rho, +\infty)$,
$$
\lim\limits_{|x|\rightarrow + \infty} v(x) |x|^{\frac{N-1}{2}}\exp \int_{\rho_0}^{|x|}\sqrt{W} = 1.
$$
\end{lemma}

\begin{proof}[\textbf{Proof of Theorem \ref{decaimentosol}}] Here we follows the lines from the proof of \cite[Proposition 6.3]{Moroz2013}, also supported by \cite[Lemma 2.2]{Lieb1983}.

    \medbreak
    
     \noindent {\bf Case $p>2$.} Let $(u, v)$ be a ground state solution for \eqref{P}. From Theorems \ref{t3} and \ref{t5}, $u$ and $v$ are positive, radially symmetric and radially decreasing (with respect to zero). In the next lines we prove the decay at infinity for $u$.
     
     From \cite[Lemma 2.2]{Lieb1983}, 
     \begin{equation}\label{nucdecres}
I_{\alpha}\ast|v|^{q} \ \ \text{and} \ \ I_{\alpha}\ast |u|^p \ \ \text{are radially symmetric and radially decreasing with respect to zero.}
     \end{equation} 
    Since $u\in L^{r}(\mathbb{R}^N)$ (Theorem \ref{t10novo}) and is radially decreasing, $I_{\alpha}\ast|v|^{q} \in L^{\beta}(\mathbb{R^N})$ (see \eqref{eq13} with $t=\theta_2q$) and is radially decreasing, it follows that
    $$
    \lim\limits_{|x|\rightarrow + \infty} (I_\alpha \ast |v|^q)(x)|u(x)|^{p-2}=0.
    $$
    Thus, there exists $\rho > 0$ such that, for $x \in \mathbb{R}^N$ with $|x|\geq \rho$, 
    $
    \dfrac{2p}{p+q}(I_\alpha \ast |v|^q)(x)|u(x)|^{p-2} \leq \dfrac{3}{4}.
    $
    Consequently, from system \eqref{P},
    $$
    -\Delta u + \dfrac{1}{4}u \leq 0.
    $$
    From Lemma \ref{lpm2} with $W=\frac{1}{4}$, let $v \in C^2(\mathbb{R}^N \setminus B_\rho, \mathbb{R})$ be such that
    $$
    \left\{ \begin{array}{lll}
        -\Delta v + \dfrac{1}{4}v = 0, \mbox{ \ if \ } x \in \mathbb{R}^N \setminus B_\rho, \vspace{5pt}\\
        v(x) = u(x),  \ \   \mbox{ \ if \ } x \in \partial B_\rho, \ \   \mbox{ \ and \ }\lim\limits_{|x|\rightarrow + \infty}v(x) = 0.
    \end{array} \right.
    $$
    Lemma \ref{lpm2} also guarantees that, for some positive constant $\mu$,
    $$
    v(x) \leq \dfrac{\mu}{|x|^{\frac{N-1}{2}}}e^{-\frac{|x|}{2}}, \qquad \forall \, x \in \mathbb{R}^N \setminus B_\rho.
    $$
    Therefore, from the comparison principle,
    $$
    u(x)\leq v(x) \leq \dfrac{\mu}{|x|^{\frac{N-1}{2}}}e^{-\frac{|x|}{2}}, \qquad \forall \, x \in \mathbb{R}^N \setminus B_\rho.
    $$
    So, there exists a real number $\nu$ such that, for every $x \in \mathbb{R}^N \setminus B_\rho$,
    $$
    \dfrac{2p}{p+q}(I_\alpha \ast |v|^q)(x)|u(x)|^{p-2}\leq \nu e^{-\frac{p-2}{2}|x|}\,, 
    $$
    which implies (again from system \eqref{P})
    $$
    -\Delta u + u \geq 0 \geq -\Delta u + W u \quad \text{in} \ \  \mathbb{R}^N \setminus B_\rho,
    $$
    where $W \in C^1 (\mathbb{R}^N \setminus B_\rho)$ is given by 
    $$
    W(x) = 1 - \nu e^{-\frac{p-2}{2}|x|} \qquad \forall \, x\in \mathbb{R}^N \setminus B_\rho.
    $$
    For this $W$, consider now $\underline{u}, \overline{u}\in C^2(\mathbb{R}^N \setminus B_\rho, \mathbb{R})$ solutions of
    $$
    \left\{
    \begin{array}{lll}
         -\Delta \underline{u}+ \underline{u}= 0, \mbox{ \ if \ } x \in \mathbb{R}^N \setminus B_\rho, \vspace{5pt}\\
        \underline{u}(x) = u(x) \ \ \ ,  \mbox{ \ if \ } x \in \partial B_\rho, \vspace{5pt}\\
        \lim\limits_{|x|\rightarrow + \infty}\underline{u} = 0,
    \end{array}
    \right.
  \qquad \text{and} \qquad
    \left\{
    \begin{array}{lll}
         -\Delta \overline{u}+ W\overline{u}= 0, \mbox{ \ if \ } x \in \mathbb{R}^N \setminus B_\rho, \vspace{5pt}\\
        \overline{u}(x) = u(x) \ \ \ \ \ \ \ ,  \mbox{ \ if \ } x \in \partial B_\rho, \vspace{5pt}\\
        \lim\limits_{|x|\rightarrow + \infty}\overline{u} = 0.
    \end{array}
    \right.
    $$
    Then, by the comparison principle, $\underline{u}\leq u \leq \overline{u}$ in $\mathbb{R}^N \setminus B_\rho$ and, from Lemma \ref{lpm2}, it follows that
    \begin{align}\label{eqlim}
        0 & < \lim\limits_{|x|\rightarrow +\infty}\underline{u}(x)|x|^{\frac{N-1}{2}}e^{|x|} \leq \liminf\limits_{|x|\rightarrow + \infty} u(x)|x|^{\frac{N-1}{2}}e^{|x|} \nonumber \\
        &\leq \limsup\limits_{|x|\rightarrow + \infty} u(x)|x|^{\frac{N-1}{2}}e^{|x|} \leq \lim\limits_{|x|\rightarrow +\infty}\overline{u}(x)|x|^{\frac{N-1}{2}}e^{|x|}<+\infty.
    \end{align}
    It remains to show that $u(x)|x|^{\frac{N-1}{2}}e^{|x|}$ has a limit as $|x|\to \infty$. Since $u$ and $\underline{u}$ are both radial and radially decreasing functions, by the comparison principle,
    $$
    \dfrac{u(s)}{\underline{u}(s)} \leq \dfrac{u(t)}{\underline{u}(t)}, \quad \forall  s\leq t, \ \ \text{with} \ \ s,t \in (\rho, +\infty), 
    $$
    that is, the function $\frac{u}{\underline{u}}$ is nondecreasing. Moreover, equation \eqref{eqlim} shows that $\dfrac{u}{\underline{u}}$ is bounded. Therefore, $\dfrac{u}{\underline{u}}$ has a finite limit and 
    $$
    \lim\limits_{|x|\rightarrow +\infty} u(x)|x|^{\frac{N-1}{2}}e^{|x|} =  \lim\limits_{|x|\rightarrow +\infty} \dfrac{u(x)}{\underline{u}(x)} \lim\limits_{|x|\rightarrow +\infty} \underline{u}(x)|x|^{\frac{N-1}{2}}e^{|x|} \in (0, +\infty).
    $$
    
    \medbreak
    
     \noindent {\bf Case $q>2$.} The same arguments are used to prove the decay at infinity for $v$. 
    \end{proof}

\subsection{Proof of Theorem \ref{decaimentosol} for the cases $p=2$ or $q=2$.}\label{sec2}
We start recalling \cite[Lemma 6.2]{Moroz2013}.

\begin{lemma}\label{auxiliar} 
Let $\alpha\in(0, N)$, $\beta \in (N, +\infty)$ and $f\in L^\infty(\mathbb{R}^N)$. If
$$
\sup\limits_{x\in \mathbb{R}^N} |f(x)||x|^\beta < +\infty,
$$
then there exists a constant $\overline{K}>0$ such that for every $x\in \mathbb{R}^N$
$$
\left| \int_{\mathbb{R}^N} \dfrac{f(y)}{|x-y|^{N-\alpha}} dy - \dfrac{1}{|x|^{N-\alpha}}\int_{\mathbb{R}^N} f(y) dy \right| \leq \dfrac{\overline{K}}{|x|^{N-\alpha}}\left(\dfrac{1}{1+|x|}+\dfrac{1}{1+|x|^{\beta -N}}\right).
$$
\end{lemma}

\medbreak 
Next we provide an asymptotic decay for $I_\alpha\ast |u|^p$ and $I_\alpha\ast |v|^q$.

\begin{lemma}\label{decaimento}
Let $N\geq 1$, $\alpha\in (0, N)$, $p, q$ satisfying \eqref{H1}, $r^\ast$ and $h^\ast$ as defined in Theorem \ref{t10novo} and $(u, v)$ be a ground state solution of \eqref{P}.
\begin{itemize}
    \item[(a)] If 
    $$
    q \geq \dfrac{\alpha}{N}p + 1 \mbox{ \ \ \ and \ \ \ } q \leq \dfrac{N}{\alpha}p - \dfrac{N}{\alpha},
    $$
    then for all $0<\varepsilon < \min\{p-1, q-1\}$, there are positive constants $K_1, K_2\in \mathbb{R}$ such that 
    \begin{equation}\label{decnuc1}
        \left| I_\alpha\ast |u|^p(x)-I_\alpha(x)\int_{\mathbb{R}^N} |u|^p \right| \leq \dfrac{K_1}{|x|^{N-\alpha}}\left(\dfrac{1}{1+|x|}+\dfrac{1}{1+|x|^{\frac{pN}{1+\varepsilon}-N}}\right) \qquad \text{and}
    \end{equation} 
    \begin{equation}\label{decnuc2}
      \left| I_\alpha\ast |v|^q(x)-I_\alpha(x)\int_{\mathbb{R}^N} |v|^q \right| \leq \dfrac{K_2}{|x|^{N-\alpha}}\left(\dfrac{1}{1+|x|}+\dfrac{1}{1+|x|^{\frac{qN}{1+\varepsilon}-N}}\right).  
    \end{equation}

    \item[(b)]  If
    $$
    \dfrac{2N}{2N-\alpha}<p,q< \dfrac{N}{N-\alpha},
    $$
    then for all $0<\varepsilon< \min\{p - r^{\ast}, q- h^{\ast}\}$, there exists positive constants $K_3, K_4\in \mathbb{R}$ such that
    \begin{equation}\label{decnuc3}
        \left| I_\alpha\ast |u|^p(x)-I_\alpha(x)\int_{\mathbb{R}^N} |u|^p \right| \leq \dfrac{K_3}{|x|^{N-\alpha}}\left(\dfrac{1}{1+|x|}+\dfrac{1}{1+|x|^{\frac{pN}{r^\ast +\varepsilon}-N}}\right) \qquad \text{and}
    \end{equation} 
    \begin{equation}\label{decnuc4}
      \left| I_\alpha\ast |v|^q(x)-I_\alpha(x)\int_{\mathbb{R}^N} |v|^q \right| \leq \dfrac{K_4}{|x|^{N-\alpha}}\left(\dfrac{1}{1+|x|}+\dfrac{1}{1+|x|^{\frac{qN}{h^\ast + \varepsilon}-N}}\right).  
    \end{equation}
    \item[(c)] If
     $$
    q \geq \dfrac{N}{N-\alpha},  \ \ \ \ q > \dfrac{N}{\alpha}p - \dfrac{N}{\alpha} \mbox{ \ \ \ and \ \ \ } \frac{2-p}{N-\alpha}<\frac{pq}{N+\alpha p},
    $$
    then for all $0<\varepsilon <  q-1$, there exists a positive constants $K_5\in \mathbb{R}$ such that 
    \begin{equation}\label{decnuc6}
      \left| I_\alpha\ast |v|^q(x)-I_\alpha(x)\int_{\mathbb{R}^N} |v|^q \right| \leq \dfrac{K_5}{|x|^{N-\alpha}}\left(\dfrac{1}{1+|x|}+\dfrac{1}{1+|x|^{\frac{qN}{1+\varepsilon}-N}}\right).  
    \end{equation}
    \item[(d)] If
      $$
    q < \dfrac{\alpha}{N}p + 1, \ \ \ \  p \geq \dfrac{N}{N-\alpha} \mbox{ \ \ \ and \ \ \ } \frac{2-q}{N-\alpha}<\frac{pq}{N+\alpha q},
    $$
    then for all $0<\varepsilon <  p-1$, there exists a positive constants $K_6 \in \mathbb{R}$ such that 
    \begin{equation}\label{decnuc7}
        \left| I_\alpha\ast |u|^p(x)-I_\alpha(x)\int_{\mathbb{R}^N} |u|^p \right| \leq \dfrac{K_6}{|x|^{N-\alpha}}\left(\dfrac{1}{1+|x|}+\dfrac{1}{1+|x|^{\frac{pN}{1+\varepsilon}-N}}\right).
    \end{equation} 
\end{itemize}
\end{lemma}

\begin{proof}
    Let $w\geq0$ be a radial and radially decreasing function, and suppose that $w\in L^\tau (\mathbb{R}^N)$ for some $\tau > 1$. If $\omega_N$ denotes the volume of the unit ball in $\mathbb{R}^N$, then
    \begin{align*}
        w(x) & \leq \dfrac{1}{\omega_N |x|^N} \int_{B_{|x|}} |w| dy \leq \dfrac{1}{\omega_N |x|^N} \left(\int_{B_{|x|}} |w|^\tau dy\right)^{\frac{1}{\tau}}\left(\omega_N |x|^N\right)^{\frac{\tau -1}{\tau}} \leq \dfrac{1}{\omega_{N}^{\frac{1}{\tau}}}\dfrac{1}{|x|^{\frac{N}{\tau}}}\left(\int_{\mathbb{R}^N} |w|^\tau dy \right)^{\frac{1}{\tau}}.
    \end{align*}
    Then, given $t>0$,
    \[
w^t(x) |x|^{\frac{tN}{\tau}} \leq \dfrac{1}{\omega_N^{\frac{t}{\tau}}}\left(\int_{\mathbb{R}^N} |w|^\tau dy \right)^{\frac{t}{\tau}}.
    \]
Next, this last inequality is used for the cases ``$w=u, \,t=p$'' and ``$w=v, \, t=q$''. In both cases, in order to apply Lemma \ref{auxiliar}, it is necessary to have $\frac{t}{\tau}>1$ and it is convenient to make it as large as possible.

\medbreak
    \noindent {\bf Case (a).} From Theorem \ref{t10novo} (a), for all $0 < \varepsilon < \min\{p-1, q-1\}$, one can choose $\tau = 1+\varepsilon$, for $u$ and for $v$. Then \eqref{decnuc1} and \eqref{decnuc2} follow from Lemma \ref{auxiliar}.

\medbreak
    \noindent {\bf Case (b).} Let $r^*$ and $h^*$ be as in Theorem \ref{t10novo} (b). Then observe that 
    \[
      \dfrac{p}{r^*}>1  \ \ \text{and} \ \ \dfrac{q}{h^*}>1.
    \]
Let $0<\varepsilon< \min\{p - r^{\ast}, q- h^{\ast}\}$ and choose $\tau = r^*+ \varepsilon$ for $u$ and $\tau = h^{\ast}+ \varepsilon$ for $v$. Then \eqref{decnuc3} and \eqref{decnuc4}  follows from Lemma \ref{auxiliar}.

\medbreak
    \noindent {\bf Case (c).} From Theorem \ref{t10novo} (c), for all $0 < \varepsilon < q-1$, one can choose $\tau = 1+\varepsilon$ for $v$. Then \eqref{decnuc6} follow from Lemma \ref{auxiliar}.

\medbreak
    \noindent {\bf Case (d).} From Theorem \ref{t10novo} (d), for all $0 < \varepsilon < p-1$, one can choose $\tau = 1+\varepsilon$ for $u$. Then \eqref{decnuc7} follow from Lemma \ref{auxiliar}.
\end{proof}

\begin{proof}[\textbf{Proof of Theorem \ref{decaimentosol}}] Let us consider the case $p=2$. First observe that, from Lemma \ref{decaimento} (a), (b) or (c), there exist positive constants $\mu$, $\sigma$ and $\mathcal{A}$ such that
    $$
    -\dfrac{\mu}{|x|^{N-\alpha+\sigma}} \leq \dfrac{2p}{p+q}(I_\alpha \ast |v|^q)(x) - \dfrac{\mathcal{A}^{N-\alpha}}{|x|^{N-\alpha}}\leq \dfrac{\mu}{|x|^{N-\alpha+\sigma}}, \quad \forall \, x \in \mathbb{R}^N,
$$
where the value $\sigma>0$ is given ahead, for each case. It is proved ahead, at Section \ref{secsub} that these inequalities also hold in the region (d) from Lemma \ref{decaimento}, but with a different argument, for which it is needed to impose $q > \frac{2N}{2N-\alpha}$. Notice that in region (d), $p=2$ implies $q<2$.

Consider the auxiliary functions $\underline{W}, \overline{W}\in C^1(\mathbb{R}^N \setminus \{0\})$, defined by
    $$
    \underline{W}(x) = 1 - \dfrac{\mathcal{A}^{N-\alpha}}{|x|^{N-\alpha}} + \dfrac{\mu}{|x|^{N-\alpha+\sigma}} \quad \mbox{ and } \quad \overline{W}(x) = 1 - \dfrac{\mathcal{A}^{N-\alpha}}{|x|^{N-\alpha}} - \dfrac{\mu}{|x|^{N-\alpha+\sigma}}.
    $$
    Then, from \eqref{P},
    $$
    -\Delta u + \overline{W}u \leq 0 \leq -\Delta u + \underline{W}u \ \ \text{in} \ \mathbb{R}^N \setminus \{0\}.
    $$
Let $\underline{u}, \overline{u}\in C^2(\mathbb{R}^N \setminus \{B_1\})$ be solutions of 
    $$
    \left\{
    \begin{array}{lll}
         -\Delta \underline{u}+ \underline{W}\, \underline{u}= 0, \mbox{ \ if \ } x \in \mathbb{R}^N \setminus B_1, \vspace{5pt}\\
        \underline{u}(x) = u(x) \ \ \ ,  \mbox{ \ if \ } x \in \partial B_1, \vspace{5pt}\\
        \lim\limits_{|x|\rightarrow + \infty}\underline{u} = 0,
    \end{array}
    \right.
   \qquad \text{and} \qquad
    \left\{
    \begin{array}{lll}
         -\Delta \overline{u}+ \overline{W}\overline{u}= 0, \mbox{ \ if \ } x \in \mathbb{R}^N \setminus B_1, \vspace{5pt}\\
        \overline{u}(x) = u(x) \ \ \ \ \ \ \ ,  \mbox{ \ if \ } x \in \partial B_1, \vspace{5pt}\\
        \lim\limits_{|x|\rightarrow + \infty}\overline{u} = 0.
    \end{array}
    \right.
    $$
Hence, from Lemma \ref{lpm2}, for $\rho$ large enough,
    $$
    \lim\limits_{|x|\rightarrow +\infty}\!\underline{u}(x)|x|^{\frac{N-1}{2}} exp\int_{\rho}^{|x|}\!\!\!\sqrt{\underline{W}(t)} dt \in (0, +\infty) \ \ \text{and} \ \   \lim\limits_{|x|\rightarrow +\infty}\!\overline{u}(x)|x|^{\frac{N-1}{2}}exp\int_{\rho}^{|x|}\!\!\!\sqrt{\overline{W}(t)} dt \in (0, +\infty).
    $$
    In particular, this implies that $\overline{u} \in L^1(\mathbb{R}^N)$. Since $0 \leq u \leq \overline{u}$, we infer that $u \in L^1(\mathbb{R}^N)$.

In what follows, the condition $N-\alpha + \sigma >1$ is required. So, the next step is to show that it holds.

\medbreak
\noindent {\bf Case $q>2$.} In this case, as proved at Section \ref{duasquad}, the decay at infinity for $v$ guarantees that $v\in L^1(\mathbb{R}^N)$. Going back to the start of the proof of Lemma \ref{decaimento} (since $(q-1)N >1$), it follows that one can take $\sigma=1$, and then the condition in verified.

\medbreak 
\noindent {\bf Case $q=2$.} In this case, as explained few lines above (for $u$), $v\in L^1(\mathbb{R}^N)$. Hence, going back to the start of the proof of Lemma \ref{decaimento} (since $(q-1)N= N \geq 1$), it follows that one can take $\sigma=1$, and then the condition in verified.

\medbreak 
\noindent {\bf Case $q<2$.} First, observe that in this case $(p,q)$ cannot be at region (c), since $q>p$ at region (c). As proved ahead at Section \ref{secsub}, in any region (a), (b) or (d)
\[
\lim \limits_{|x| \to \infty} |v(x)|^q |x|^{\frac{q(N-\alpha)}{2-q}} \in (0, \infty).
\]
So, in order to apply Lemma \ref{auxiliar}, it must hold that
\[
\dfrac{q(N-\alpha)}{2-q} > N, \ \ \text{that is,} \ \ q> \dfrac{2N}{2N-\alpha},
\]
and then $\sigma = \min\{ 1, \frac{q(N-\alpha)}{2-q} - N\}$. Therefore, it is necessary to require
\[
\dfrac{q(N-\alpha)}{2-q} - \alpha > 1,  \ \ \text{that is,} \ \ q >\dfrac{2(\alpha +1)}{N + 1}.
\]

Next, it is shown that this inequality holds at regions (a), (b) or (d).

\noindent {\bf Region (a)}. In this case, since $p \neq q$, it follows that $q> \frac{N}{N-\alpha}$ and, from direct calculation, one sees that $\frac{N}{N-\alpha} \geq \frac{2(\alpha +1)}{N + 1}$.

\noindent {\bf Regions (b) and (d)}.
In this cases  the condition $q> \max\left\{ \frac{2N}{2N-\alpha}, \frac{2(\alpha +1)}{N + 1} \right\}$ is imposed.

\medbreak    

    Now, since (here the condition $N-\alpha + \sigma >1$ enters)
    $$
    \int_{\rho}^{+\infty} \sqrt{\underline{W}} - \sqrt{\overline{W}} dt < +\infty, 
    $$
it follows that
   \begin{align*}
    &\lim\limits_{|x|\rightarrow +\infty}\underline{u}(x)|x|^{\frac{N-1}{2}}exp\int_{\rho}^{|x|}\sqrt{1-\dfrac{\mathcal{A}^{N-\alpha}}{t^{N-\alpha}}} dt \in (0, +\infty) \qquad{and} \\
    &\lim\limits_{|x|\rightarrow +\infty}\overline{u}(x)|x|^{\frac{N-1}{2}}exp\int_{\rho}^{|x|}\sqrt{1-\dfrac{\mathcal{A}^{N-\alpha}}{t^{N-\alpha}}} dt \in (0, +\infty).
   \end{align*}
    Therefore, repeating the arguments from Section \ref{duasquad}, we infer that
    $$
     \lim\limits_{|x|\rightarrow +\infty}u(x)|x|^{\frac{N-1}{2}}exp\int_{\rho}^{|x|}\sqrt{1-\dfrac{\mathcal{A}^{N-\alpha}}{t^{N-\alpha}}} dt \in (0, +\infty),  
    $$
    as desired.
\end{proof}

\subsection{Proof of Theorem \ref{decaimentosol} for the cases $p<2$ or $q<2$.} \label{secsub}

We start recalling \cite[Lemma 6.7]{Moroz2013}.
\begin{lemma}\label{lpmenor2}
Let $\rho > 0$, $\beta >0$, $\gamma > 0$ and $u\in C^2(\mathbb{R}^N \setminus B_\rho)$. If for every $x\in \mathbb{R}^N \setminus B_\rho$, 
$$
-\Delta u +\lambda u = \dfrac{1}{|x|^\beta} \qquad \text{and} \qquad
\lim\limits_{|x|\rightarrow +\infty}u(x)=0,
$$
then
$$
\lim\limits_{|x|\rightarrow +\infty}u(x)\lambda |x|^\beta = 1.
$$
\end{lemma}

\medbreak
\begin{proof}[\textbf{Proof of Theorem \ref{decaimentosol}.}] Let us consider the case $p<2$. The key argument is to prove that the kernel $I_{\alpha}\ast |v|^q$ has a good decay at infinity, namely
\begin{equation}\label{desig}
\left| I_\alpha \ast |v|^q (x) - I_\alpha(x)\int_{\mathbb{R}^N} |v|^q \right| \leq \dfrac{\mu}{|x|^{N-\alpha+\delta}},
\end{equation}
for some positive $\mu$ and $\delta$.

\medbreak
\noindent {\bf Case $q>2$.} In this case, as proved at Section \ref{duasquad}, the decay at infinity for $v$ guarantees that $v\in L^1(\mathbb{R}^N)$. Then, going back to the start of the proof of Lemma \ref{decaimento} (since $(q-1)N >1$), it follows that one can take $\delta=1$, and then the condition in verified.

\medbreak
\noindent {\bf Case $q=2$.} First observe that in this case the pair $(p,q)$ cannot be at region (d) from Lemma \ref{decaimento}. Also observe that, at region (c), $p> \frac{2N}{2N-\alpha}$ implies $\frac{2-p}{N-\alpha}< \frac{pq}{N+\alpha p}$. So, condition \eqref{desig} follows from Lemma \ref{decaimento}.

\medbreak
\noindent {\bf Case $q<2$.} When $(p,q)$ belongs to one of the regions (a), (b) or (c) we argue as in the previous paragraph to obtain \eqref{desig} from Lemma \ref{decaimento}. When $(p,q)$ is at region (d), one should restart this proof for $v$, using Lemma \ref{decaimento} (d) to obtain the decay for $I_{\alpha}\ast |u|^p$, then to obtain the polynomial decay for $v$ (as ahead), namely
\[
\lim \limits_{|x| \to \infty} |v(x)|^q |x|^{\frac{q(N-\alpha)}{2-q}} \in (0, \infty).
\]
From this, then we apply Lemma \ref{auxiliar} with $\beta = \frac{q(N-\alpha)}{2-q}$, and observe that $\beta>N$, since $q> \frac{2N}{2N-\alpha}$.

\medbreak
From Theorems \ref{t3} and \ref{t10novo}, $u>0$ in $\mathbb{R}^N$ and $u\in C^2(\mathbb{R}^N)$. Hence, by the chain rule, $u^{2-p}\in C^2(\mathbb{R}^N)$ and
$$
-\Delta u^{2-p} = -(2-p)u^{1-p}\Delta u +(2-p)(p-1)|\nabla u|^2 \quad \text{in} \quad \mathbb{R}^N.
$$
Since $p\in \left(\max\left\{1, \frac{2\alpha}{N}\right\}, 2\right)$ and $(u, v)$ is a solution of \eqref{P}, from \eqref{desig}, for every $x\in \mathbb{R}^N$,
$$
-\Delta u^{2-p}(x)+(2-p)u^{2-p}(x) \geq (2-p)\dfrac{2p}{p+q}I_\alpha (x) \left(\int_{\mathbb{R}^N} |v|^q dx -\dfrac{\mu}{|x|^\delta}\right).
$$
Now, consider $\underline{u}:\mathbb{R}^N \setminus B_1 \rightarrow \mathbb{R}$ a solution of
$$
\left\{
\begin{array}{lll}
     -\Delta \underline{u} + (2-p)\underline{u} = (2-p)\dfrac{2p}{p+q}I_\alpha (x) \left(\displaystyle\int_{\mathbb{R}^N} |v|^q dx -\dfrac{\mu}{|x|^\delta}\right), \mbox{ \ in \ } \mathbb{R}^N \setminus B_1, \vspace{5pt}\\
     \underline{u}(x)=u^{p-2}(x), \ \ \ \ \ \ \ \ \ \ \ \ \ \ \ \ \ \ \ \ \ \ \ \ \ \ \ \ \ \ \ \ \ \ \ \ \ \ \ \ \ \ \ \ \ \ \ \ \ \ \ \ \ \ \ \ \ \ \ \ \ \ \ \ \ \ \ \ \ \ \ \ \mbox{ \ in \ } \partial B_1, \vspace{5pt}\\
     \lim\limits_{|x|\rightarrow +\infty}\underline{u}(x)=0.
\end{array}
\right.
$$
Applying the linearity of the operator $-\Delta +1$ and Lemma \ref{lpmenor2} to each equation with $\lambda=2-p$ and $\beta=N-\alpha$, and, $\lambda=2-p$ and $\beta=N-\alpha+\delta$, concluding that
$$
\lim\limits_{|x|\rightarrow +\infty}\dfrac{\underline{u}(x)}{I_\alpha(x)}= \dfrac{2p}{p+q}\int_{\mathbb{R}^N}|v|^q dx.
$$
By the comparison principle, $\underline{u}\leq u^{2-p}$ in $\mathbb{R}^N \setminus B_1$, which implies
\begin{equation}\label{subsol}
    \liminf\limits_{|x|\rightarrow +\infty}\dfrac{u^{2-p}(x)}{I_\alpha(x)}\geq \dfrac{2p}{p+q}\int_{\mathbb{R}^N} |v|^q dx.
\end{equation}

On the other hand, applying the Young's inequality to the kernel, we have
$$
\dfrac{2p}{p+q}(I_\alpha \ast |v|^q)|u|^{p-2}u\leq \left(\dfrac{2p}{p+q}\right)^{\frac{1}{2-p}}(2-p)(I_\alpha \ast |v|^q)^{\frac{1}{2-p}}+(p-1)u. 
$$
From \eqref{desig} and the mean value theorem, for $x\in \mathbb{R}^N \setminus B_1$, it follows that
$$
(I_\alpha \ast |v|^q)^{\frac{1}{2-p}}(x) \leq I_{\alpha}^{\frac{1}{2-p}}(x)\left( \left(\int_{\mathbb{R}^N}|v|^q dx \right)^{\frac{1}{2-p}}+\dfrac{\nu}{|x|^\delta}\right).
$$
Hence, for every $x\in \mathbb{R}^N \setminus B_1$, we infer that
$$
    -\Delta u + (2-p)u  \leq \left(\dfrac{2p}{p+q}\right)^{\frac{1}{2-p}}(2-p)I_{\alpha}^{\frac{1}{2-p}}(x)\left( \left(\int_{\mathbb{R}^N}|v|^q dx \right)^{\frac{1}{2-p}}+\dfrac{\nu}{|x|^\delta}\right).
$$
Consider, now, $\overline{u}\in C^2(\mathbb{R}^N \setminus B_1)$ a solution 
$$
\left\{
\begin{array}{lll}
     -\Delta \overline{u} + (2-p)\overline{u} = \left(\dfrac{2p}{p+q}\right)^{\frac{1}{2-p}}(2-p)I_{\alpha}^{\frac{1}{2-p}}(x)\displaystyle\left( \left(\int_{\mathbb{R}^N}|v|^q dx \right)^{\frac{1}{2-p}}+\dfrac{\nu}{|x|^\delta}\right), \mbox{ \ in \ } \mathbb{R}^N \setminus B_1, \vspace{5pt}\\
     \overline{u}(x)=u(x), \ \ \ \ \ \ \ \ \ \ \ \ \ \ \ \ \ \ \ \ \ \ \ \ \ \ \ \ \ \ \ \ \ \ \ \ \ \ \ \ \ \ \ \ \ \ \ \ \ \ \ \ \ \ \ \ \ \ \ \ \ \ \ \ \ \ \ \ \ \ \ \ \ \ \ \ \ \ \ \ \ \ \ \ \ \ \  \mbox{ \ in \ } \partial B_1, \vspace{5pt}\\
     \lim\limits_{|x|\rightarrow +\infty}\overline{u}(x)=0.
\end{array}
\right.
$$
Once again, using the linearity of the operator $-\Delta +1$ and applying the Lemma \ref{lpmenor2} to each of the new equations, we conclude that
$$
\lim\limits_{|x|\rightarrow +\infty} \dfrac{\overline{u}(x)}{I_{\alpha}^{\frac{1}{2-p}}(x)}= \left(\dfrac{2p}{p+q}\right)^{\frac{1}{2-p}} \left(\int_{\mathbb{R}^N}|v|^q dx \right)^{\frac{1}{2-p}}.
$$
Therefore, by the comparison principle, $u\leq \overline{u}$ in $\mathbb{R}^N \setminus B_1$ and
$$
\limsup\limits_{|x|\rightarrow +\infty}\dfrac{u^{2-p}(x)}{I_\alpha(x)}\leq \dfrac{2p}{p+q}\int_{\mathbb{R}^N}|v|^q dx,
$$
concluding the proof. 
\end{proof}

\section{A non-existence result}\label{sec-NE}

In this section we establish the region of a $(p, q)$-plan where no nontrivial solution for \eqref{P} exists. To achieve this aim, we will use a Pohoz\v{a}ev-type identity.

\begin{proposition}\label{p1}
Let $0 < \alpha< N$, $p>1, q>1, \theta_1>1, \theta_2> 1$ be such that $\frac{1}{\theta_1}+ \frac{1}{\theta_ 2} = \frac{N+\alpha}{N}$. If $(u, v)$ is a weak solution of \eqref{P} such that $u\in H^{1}(\mathbb{R}^N) \cap W_{loc}^{2, 2}(\mathbb{R}^N)\cap W^{1, \theta_1 p}(\mathbb{R}^N)$ and $v\in H^{1}(\mathbb{R}^N)\cap W_{loc}^{2, 2}(\mathbb{R}^N)\cap W^{1, \theta_2 q}(\mathbb{R}^N)$, then
$$
\dfrac{N-2}{2}\int_{\mathbb{R}^N} (|\nabla u|^2+|\nabla v|^2) dx + \dfrac{N}{2}\int_{\mathbb{R}^N} u^2 + v^2 dx = 2\dfrac{N+\alpha}{p+q}\int_{\mathbb{R}^N} (I_\alpha \ast |u|^p)|v|^q dx. 
$$
\end{proposition} 
\begin{proof}
The proof can be done similarly as \cite[Proposition 3.1]{Moroz2013}, considering $\varphi\in C_{c}^{1}(\mathbb{R}^N)$ such that $\varphi \equiv 1$ in $B_1$ and defining, for $\lambda\in (0, +\infty)$ and $x\in \mathbb{R}^N$, $f_\lambda \in W^{1, 2}(\mathbb{R}^N)\cap L^{1, \theta_1 p}(\mathbb{R}^N)$ and $g_\lambda \in W^{1, 2}(\mathbb{R}^N)\cap L^{1, \theta_2 q}(\mathbb{R}^N)$ by
$$
f_\lambda(x)=\varphi(\lambda x)x\cdot \nabla u(x) \mbox{ \ \ \ and \ \ \ } g_\lambda(x)=\varphi(\lambda x)x\cdot \nabla v(x),
$$
respectively.
\end{proof}

\begin{proof}[{\bf Proof of Theorem \ref{t4}}]
From the fact that $(u, v)$ is a weak solution of \eqref{P} and applying Proposition \ref{p1}, we infer that
$$
\left(\dfrac{N-2}{2}-\dfrac{N+\alpha}{p+q}\right)\int_{\mathbb{R}^N} |\nabla u|^2 + |\nabla v|^2 dx + \left(\dfrac{N}{2}-\dfrac{N+\alpha}{p+q}\right)\int_{\mathbb{R}^N} u^2 + v^2 dx = 0.
$$
From this identity, if $(p+q)(N-2) \geq 2(N+\alpha)$, then $\dfrac{N}{2}-\dfrac{N+\alpha}{p+q}>0$, which imply that $u=v=0$. On the other hand, if $(p+q)N \leq 2(N+\alpha)$, then $\dfrac{N-2}{2}-\dfrac{N+\alpha}{p+q}<0$, which imply that  $\nabla u=\nabla v=0$, and then $u=v=0$.
\end{proof}

\noindent{\bf Data availability} All data generated or analyzed during this study are included in this article.

\medbreak

\noindent{\bf Acknowledgements.} Eduardo de Souza B\"oer is supported by FAPESP/Brazil grant 2023/05445-2. \linebreak Ederson Moreira dos Santos is partially supported by CNPq/Brazil grant 312867/2023-9 and \linebreak FAPESP/Brazil grant 2022/16407-1.

\vspace{1cm}
\noindent\textsc{Eduardo de Souza B\"oer and Ederson Moreira dos Santos}\\
Instituto de Ci\^encias Matem\'aticas e de Computa\c c\~ao\\
Universidade de S\~ao Paulo -- USP\\
13566-590, Centro, S\~ao Carlos - SP, Brazil\\
\noindent\texttt{eduardoboer@usp.br, ederson@icmc.usp.br}.

\end{document}